\documentclass[12pt,a4paper]{article}%
\usepackage[utf8]{inputenc}
\usepackage{hyperref}
\usepackage{amsmath}
\usepackage{amsfonts}
\usepackage{amssymb}
\usepackage{graphicx}%
\providecommand{\U}[1]{\protect\rule{.1in}{.1in}}
\newtheorem{theorem}{Theorem}

\newtheorem{corollary}[theorem]{Corollary}

\newtheorem{definition}[theorem]{Definition}

\newtheorem{lemma}[theorem]{Lemma}

\newtheorem{problem}[theorem]{Problem}
\newtheorem{proposition}[theorem]{Proposition}
\newtheorem{remark}[theorem]{Remark}
\newtheorem{observation}[theorem]{Observation}

\newenvironment{proof}[1][Proof]{\noindent\textbf{#1.} }{\ \hfill \rule{0.5em}{0.5em}\bigskip}
\graphicspath{{D:/Dropbox/Riste-Sedlar/MetricWeak/Slike/}}

\begin{document}

\title{Resolving vertices of graphs with differences}
\author{Iztok Peterin$^{a,b}$, Jelena Sedlar$^{c,e}$, Riste \v{S}krekovski$^{b,d,e}$
and Ismael G. Yero$^{f}$\\{\small $^{a}$ \textit{University of Maribor, FEECS, Maribor, Slovenia }}\\[0.1cm] {\small $^{b}$ \textit{Institute of Mathematics, Physics and
Mechanics, Ljubljana, Slovenia }}\\[0.1cm] {\small $^{c}$ \textit{University of Split, Faculty of civil
engineering, architecture and geodesy, Split, Croatia }}\\[0.1cm] {\small $^{d}$ \textit{University of Ljubljana, FMF, Ljubljana,
Slovenia }}\\[0.1cm] {\small $^{e}$ \textit{Faculty of Information Studies, Novo Mesto,
Slovenia }}\\[0.1cm] {\small $^{f}$ \textit{Universidad de C\'{a}diz, Departamento de
Matem\'{a}ticas, Algeciras Campus, Spain}}}
\maketitle

\begin{abstract}
The classical (vertex) metric dimension of a graph $G$ is defined as the
cardinality of a smallest set $S\subseteq V(G)$ such that any two vertices $x$
and $y$ from $G$ have different distances to least one vertex from $S.$ The
$k$-metric dimension is a generalization of that notion where it is required
that any pair of vertices has different distances to at least $k$ vertices
from $S.$ In this paper, we introduce the weak $k$-metric dimension of a graph
$G,$ which is defined as the cardinality of a smallest set of vertices $S$
such that the sum of the distance differences from any pair of vertices to all
vertices of $S$ is at least $k.$ This dimension is "stronger" than the
classical metric dimension, yet "weaker" than $k$-metric dimension, and it can
be formulated as an ILP problem. The maximum $k$ for which the weak $k$-metric
dimension is defined is denoted by $\kappa(G).$ We first prove several
properties of the weak $k$-metric dimension regarding the presence of true or
false twin vertices in a graph. Using those properties, the $\kappa(G)$ is
found for some basic graph classes, such as paths, stars, cycles, and complete
(bipartite) graphs. We also find $\kappa(G)$ for trees and grid graphs using
the observation that the distance difference increases by the increase of the
cardinality of a set $S.$ For all these graph classes we further establish the
exact value of the weak $k$-metric dimension for all $k\leq\kappa(G).$

\end{abstract}



\section{Introduction}

This article deals with topics on metric dimension, which is a classical
parameter in graph theory, already known from 1970's decade. It is usually
said that the first contributions on this topic separately appeared in
\cite{Slater1975} and \cite{Harary1976}, although some notions about the
metric dimension are earlier known from \cite{Blumenthal}. This parameter is
nowadays very well studied and one can find a rich literature on it. Recent
surveys about this are \cite{Tillquist,Kuziak}. Moreover, some recent and
significant works on the topic are \cite{Claverol,Corregidor,Geneson,Hakanen,
Klavzar, Mashkaria,Sedlar,Wu}. In order to better proceed with the flow of our
exposition, we first present the main definitions of our work, and further
describe the motivations of it.

From now on, all graphs considered are simple and connected. The
\emph{distance} between a pair of vertices $u,v$ of a graph $G,$ denoted by
$d_{G}(u,v),$ is the length of a shortest path connecting $u$ and $v$ in $G.$
When the graph $G$ is clear from the context we will write $d(u,v)$ for short.
Given a graph $G$ and three vertices $x,y$ and $s,$ we define the
\emph{distance difference (between }$x$\emph{ and }$y$\emph{ regarding }%
$s$\emph{)}
\[
\Delta_{s}(x,y)=|d_{G}(x,s)-d_{G}(y,s)|.
\]
Now, given a set of vertices $S\subseteq V(G)$ and an integer $k\geq1$, we say
that $S$ is a \emph{weak $k$-resolving set} of $G$ if the total distance
difference over the set $S$ satisfies $\sum_{s\in S}\Delta_{s}(x,y)\geq k$ for
every pair of vertices $x,y\in V(G)$. Also, a weak $k$-resolving set of the
smallest possible cardinality in $G$ is a \emph{weak $k$-metric basis} and its
cardinality the \emph{weak} $k$-\emph{metric dimension} of $G$, which is
denoted by $\operatorname*{wdim}_{k}(G)$. Note that every weak $k$-resolving
set $S$ satisfies that $|S|\geq k$ and that, if $k>1$, then $S$ is also a weak
$(k-1)$-resolving set. Thus, $\operatorname*{wdim}_{k}(G)\leq
\operatorname*{wdim}_{k+1}(G).$

The reader can immediately observe that it is not possible to compute the weak
$k$-metric dimension of a given graph for every integer $k\geq1.$ In this
sense, for each graph $G$, we define $\kappa(G)$ as the largest positive
integer $k$ such $G$ contains a weak $k$-resolving set, and say that $G$ is
\emph{weak }$\kappa(G)$\emph{-metric dimensional}. We also say that the set of
integers $\{1,\dots,\kappa(G)\}$ is the set of \emph{suitable values} for
computing $\operatorname*{wdim}_{k}(G)$. We use notation $[n]=\{1,\ldots,n\}$
and with this the set of suitable values of a graph $G$ is denoted by
$[\kappa(G)].$

For a vertex $x\in V(G)$, the \emph{open neighborhood} $N_{G}(x)$ of $x$ is
the set of all neighbors of $x,$ and the \emph{closed neighborhood} $N_{G}[x]$
is defined as $N_{G}(x)\cup\{x\}$. As usual, the degree of $x$ is
$d_{G}(x)=\left\vert N_{G}(x)\right\vert .$ Now, two vertices $x,y$ are
\emph{false twins} if $N_{G}(x)=N_{G}(y)$, and $x,y$ are \emph{true twins} if
$N_{G}[x]=N_{G}[y]$. Two different vertices $x,y$ are \emph{twins} if they are
either false or true twins. Finally, for a set of vertices $S\subseteq V(G),$
we define the \emph{open neighborhood} $N_{G}(S)=\cup_{x\in S}N_{G}(x)$ and
the \emph{closed neighborhood} $N_{G}[S]=\cup_{x\in S}N_{G}[x].$

\section{Motivation}

The classical metric dimension of a graph $G$ searches for the cardinality of
the smallest set $S$ of vertices of $G$ such that for every two vertices
$x,y\in V(G)$ there is a vertex $s\in S$ such that $d_{G}(x,s)\neq d_{G}%
(y,s)$. Such set uniquely \textquotedblleft identifies\textquotedblright\ all
the vertices of the graph, and this property has proved to be very useful for
location problems in networks. However, although this set $S$ (through some
vertex $s\in S$) identifies the pair of vertices $x,y$, it does not give any
hint into how much \textquotedblleft large\textquotedblright\ this
\textquotedblleft identification\textquotedblright\ property is with respect
to the set $S$ (or to the vertex $s\in S$). This suggest the idea of using
some metric that would quantify this difference. There could be clearly
several possible metrics that can make so, but one might immediately think
about a typical one, which is that of a modular distance. That is, for a
vertex $s$ and two vertices $x,y$, this can be represented as the previously
mentioned $\Delta_{s}(x,y)=|d_{G}(x,s)-d_{G}(y,s)|$.

Now, in order to construct some structure that can play some significant role
into quantifying the differences mentioned above, there could be two possible
directions to consider. A first one which is more local, could be searching
for a set $S\subseteq V(G)$ such that for any two vertices $x,y\in V(G)$, it
follows that $\Delta_{s}(x,y)\geq k$ for some $s\in S$. However, this approach
turns out to be relatively useless and too restrictive. Consequently, one
might instead consider a \textquotedblleft more global\textquotedblright%
\ setting, that would involve not only a vertex of the given set $S$, but the
whole set $S$. If the quantification of differences is $\sum_{s\in S}%
\Delta_{s}(x,y)$, clearly the whole set of vertices $S$ is involved. If we
want to have a threshold for the differences, then we need to delimit a value
$k$ (clearly positive) for them. This is traduced to setting an integer
$k\geq1$ such that $\sum_{s\in S}\Delta_{s}(x,y)\geq k$ for every two vertices
$x,y$ of the graph $G$, and gives sense to considering the weak $k$-metric
dimension of graphs.

\paragraph{ILP model.}

The problem of finding a smallest weak $k$-resolving set $S$ and consequently
the weak $k$-metric dimension $\operatorname*{wdim}_{k}(G)$ of a graph $G$ can
be formulated as an integer linear programming model. Let $V(G)=\{v_{1}%
,\ldots,v_{n}\}$ be the set of vertices of a graph $G.$ For a set of vertices
$S\subseteq V(G),$ the integer variable $x_{i}$ is defined by
\[
x_{i}=\left\{
\begin{array}
[c]{ll}%
1 & \text{if }v_{i}\in S,\\
0 & \text{if }v_{i}\not \in S.
\end{array}
\right.
\]
The problem of finding a smallest weak $k$-resolving set $S$ of $G$ is now
formulated as
\begin{equation}%
\begin{tabular}
[c]{ll}
& $\min\sum_{i=1}^{n}x_{i}\medskip$\\
s.t. & $\sum\limits_{i=1}^{n}\left\vert d(v_{a},v_{i})-d(v_{b},v_{i}%
)\right\vert \cdot x_{i}\geq k$ for any $v_{a},v_{b}\in V(G),$\\
& $x_{i}\in\{0,1\}$ for $1\leq i\leq n.$%
\end{tabular}
\ \ \ \ \ \ \label{For_ILP}%
\end{equation}

\paragraph{Comparison with $k$-metric dimension.}

The concept of metric dimension in graphs was generalized to that of
$k$-metric dimension in \cite{Estrada-Moreno2013} as follows. Given an integer
$k\geq1$, a set $S$ of vertices of $G$ is a $k$-\emph{resolving set} if for
any two vertices $x,y\in V(G)$, there exist $k$ distinct vertices $v_{1}%
,\dots,v_{k}\in S$ such that $d_{G}(x,v_{i})\neq d_{G}(y,v_{i})$ for every
$i\in\{1,\dots,k\}$. Also, the $k$-\emph{metric dimension} of $G$, denoted by
$\dim_{k}(G)$ is the cardinality of a smallest possible $k$-resolving set.
Clearly, the case $k=1$ represents the classical metric dimension. Moreover,
by $\kappa^{\prime}(G)$ we represent the largest integer $k$ for which $G$
contains a $k$-resolving set, and it is said that $G$ is $\kappa^{\prime}%
(G)$\emph{-metric dimensional}.

It is now readily observed that a $k$-resolving set in a graph $G$ is also a
weak $k$-resolving set, so $\kappa^{\prime}(G)\leq\kappa(G).$ However, the
contrary is not true in general. That is, if one can find a weak $k$-resolving
set in $G$, then such a set is not necessarily a $k$-resolving set, and even
more, there could be no $k$-resolving sets for such value $k$. This gives
sense to one part of the following relationships.

\begin{remark}
\label{rem:dim-wdim} For any graph $G$ and any positive integer $k\leq
\kappa^{\prime}(G)$,
\[
\dim(G)\leq\operatorname*{wdim}\nolimits_{k}(G)\leq\dim_{k}(G).
\]

\end{remark}

\begin{proof}
The second inequality was already explained. To see the first inequality,
consider any weak $k$-resolving set $S$. Since $\sum_{s\in S}\Delta
_{s}(x,y)\geq k$ for any two arbitrary vertices $x,y$ of $G$, we deduce that
$\sum_{s\in S}\Delta_{s}(x,y)\geq1$, which means that there must exist a
vertex $s^{\prime}\in S$ such that $1\leq\Delta_{s^{\prime}}(x,y)=|d_{G}%
(x,s^{\prime})-d_{G}(y,s^{\prime})|$. Thus, we have $d_{G}(x,s^{\prime})\neq
d_{G}(y,s^{\prime})$, and so, $S$ is a resolving set, which leads to the
desired result.
\end{proof}

From Remark \ref{rem:dim-wdim} we deduce the following consequence for the
particular case $k=1$.

\begin{corollary}
\label{cor:dim-equal-wdim-k-1} For any graph $G$, $\operatorname*{wdim}%
_{1}(G)=\dim_{1}(G)=\dim(G)$.
\end{corollary}

We may also recall that $\kappa^{\prime}(G)\leq\kappa(G)$ for any graph $G$,
and the inequality can be strict in many situations. For instance, it is
already known from \cite{Estrada-Moreno2013} that any graph $G$ is $2$-metric
dimensional if and only if $G$ has at least two twins (true or false). In
contrast with this, it can be established (see Corollary
\ref{prop:false-twins}) that graphs without true twins having false twins are
weak $4$-metric dimensional.

\paragraph{Comparison with local $k$-metric dimension.}

Another parameter which is closely related to our one comes as follows. The
local metric dimension of graphs was introduced in \cite{Okamoto} and after
this several articles on the topic have been developed (see the survey
\cite{Kuziak} for more information on this fact). A \emph{local resolving set}
of a graph $G$ is a set $S$ such that for any two adjacent vertices $x,y\in
V(G)$, there is a vertex $w\in S$ such that $d_{G}(x,w)\neq d_{G}(y,w)$ ($S$
uniquely identifies pairs of adjacent vertices). The \emph{local metric
dimension}, denoted by $\operatorname{ldim}(G)$, represents the cardinality of
a smallest local resolving set of $G$.

In a similar and natural manner, as the metric dimension was generalized, the
local metric dimension can also be generalized to the \emph{local $k$-metric
dimension}, which we denote by $\operatorname{ldim}_{k}(G)$. The concept of
local $k$-resolving set follows along the lines of the generalization. As we
next show, any weak $k$-resolving set is also a local $k$-resolving set.

\begin{lemma}
If $S$ is a weak $k$-resolving set of a graph $G$, then it is also a local
$k$-resolving set.
\end{lemma}

\begin{proof}
The result follows by the fact that for any two adjacent vertices $x,y$ of $G$
and any vertex $s\in S$ it follows that the distances $d_{G}(x,s)$ and
$d_{G}(y,s)$ can differ by at most one. For a weak $k$-resolving set $S$ we
have $\sum_{s\in S}|d_{G}(x,s)-d_{G}(y,s)|\geq k$ since $S$ has at least $k$
vertices. Therefore, there are at least $k$ vertices $v_{1},\dots,v_{k}\in S$
such that $d_{G}(x,v_{i})\neq d_{G}(y,v_{i})$ for any $i\in\lbrack k]$, and
so, $S$ is a local $k$-resolving set.
\end{proof}

\begin{corollary}
For any graph $G$, it holds that $\operatorname*{ldim}_{k}(G)\leq
\operatorname*{wdim}_{k}(G)$ whenever both these dimensions are defined for
$k.$
\end{corollary}

\section{Determining $\kappa$}

\label{sec:kappa}

Since computing $\operatorname*{wdim}_{k}(G)$ can be made in the corresponding
suitable set of values of $k$, our first task must be finding the value
$\kappa(G)$. To this end, for a pair of vertices $x,y\in V(G)$ and a set
$S\subseteq V(G),$ let us denote $\Delta_{S}(x,y)=\sum_{s\in S}\Delta
_{s}(x,y)=\sum_{s\in S}\left\vert d(x,s)-d(y,s)\right\vert $. Since the
contribution of each vertex $s\in S$ to the value of $\Delta_{S}(x,y)$ is
nonnegative, we conclude that for $S\subseteq S^{\prime}\subseteq V(G)$ it
holds that $\Delta_{S}(x,y)\leq\Delta_{S^{\prime}}(x,y).$ Thus, denoting
$\Delta(x,y)=\Delta_{V(G)}(x,y)$ we immediately obtain the following observation.

\begin{observation}
\label{Obs_vertexSet}For any graph $G$, $\kappa(G)=\displaystyle\min\left\{
\Delta(x,y)\,:\,x,y\in V(G)\right\}  $.
\end{observation}

From the above observation, we notice that computing the value $\kappa(G)$ for
a given graph $G$ can be polynomially done.\ However, an explicit formula for
it seems to be hard to describe. We next give some special cases, and
particularize some situations for graphs under some particular situations, or
for some special families of graphs.

Since the entire vertex set of a graph is indeed a weak $k$-resolving set for
some $k\geq2$, it is clear that any nontrivial graph $G$ satisfies that
$\kappa(G)\geq2$, or equivalently that any graph is at least weak $2$-metric
dimensional. In this sense, it is worthy of characterizing those graphs that
are weak $2$-metric dimensional.

\begin{proposition}
\label{prop:2-met-dim-al} A nontrivial graph $G$ is weak $2$-metric
dimensional if and only if $G$ has true twins.
\end{proposition}

\begin{proof}
If $G$ has true twins $x,y$, then clearly $\Delta_{s}(x,y)=1$ for
$s\in\{x,y\},$ and $\Delta_{s}(x,y)=0$ for every $s\in V(G)\backslash\{x,y\}.$
Thus $\kappa(G)=\displaystyle\min\left\{  \Delta(x,y)\,:\,x,y\in V(G)\right\}
=2$.

On the other hand, assume $G$ is weak $2$-metric dimensional. Suppose that $G$
has no true twins and let $x,y\in V(G)$. If $x,y$ are false twins, then
$\Delta_{s}(x,y)=2$ for $s\in\{x,y\},$ and $\Delta_{s}(x,y)=0$ for every $s\in
V(G)\backslash\{x,y\},$ thus $\Delta(x,y)=4$. Moreover, if $x,y$ are not false
twins, then $N_{G}(x)\neq N_{G}(y)$ and so there exists a vertex
$s\not \in \{x,y\}$ such that $\Delta_{s}(x,y)>0,$ which means that
$\Delta(x,y)\geq3$. As a consequence, we deduce that $\kappa
(G)=\displaystyle\min\left\{  \Delta(x,y)\,:\,x,y\in V(G)\right\}  \geq3$, a
contradiction. Therefore, $G$ must have true twins.
\end{proof}

Since the graphs that are weak $2$-metric dimensional are easily described, we
may consider some larger values for $\kappa(G)$.

\begin{proposition}
\label{prop:no-3-dim-al}A graph $G$ is weak $3$-metric dimensional if and only if

\begin{itemize}
\item[(i)] $G$ has no true twins; and

\item[(ii)] there exists a pair of adjacent vertices $x$ and $y$ in $G$ such
that $N_{G}[x]\backslash\{z\}=N_{G}[y]$ for some vertex $z$ and all neighbors
of $z$ belong to $N_{G}[N_{G}[x]\cap N_{G}[y]].$
\end{itemize}
\end{proposition}

\begin{proof}
If $G$ has no true twins then Proposition \ref{prop:2-met-dim-al} implies
$\kappa(G)\geq3.$ If the condition (ii) also holds, then $x$ and $y$ are
distinguished only by $s\in\{x,y,z\}$ where $\Delta_{s}(x,y)=1$ for each such
$s.$ Notice that $z$ must be neighbor of $x.$ Hence, we conclude that
$\Delta(x,y)=3,$ so $G$ is weak $3$-metric dimensional.

Now assume that $G$ is weak $3$-metric dimensional. If $G$ has true twins,
then $G$ is weak $2$-metric dimensional by Proposition \ref{prop:2-met-dim-al}%
, and since $G$ is weak $3$-dimensional the condition (i) must hold. Also,
there must exist a pair of vertices $x,y$ in $G$ such that $\Delta(x,y)=3.$
Notice that $x$ and $y$ are adjacent, otherwise for each $s\in\{x,y\}$ it
would hold $\Delta_{s}(x,y)=2,$ which further implies $\Delta(x,y)\geq4,$ a contradiction.

As $x$ and $y$ are not true twins, we may assume that there is a vertex $z\in
V(G)$ which is adjacent to $x$ and not adjacent to $y.$ Notice that if there
are two such vertices, say $z$ and $z^{\prime},$ then the distance difference
of $x$ and $y$ would be at least $4.$ Similarly, if there is a neighbor
$z^{\prime\prime}$ of $y$ that is not adjacent to $x$ then again this
difference is at least $4.$ So, we obtain $N_{G}[x]\backslash\{z\}=N_{G}[y].$

Next, let $u$ be any neighbor of $z$ distinct from $x.$ Observe that $u$ is on
a same distance to $y$ as to $x,$ as otherwise we increase $\Delta(x,y).$ So,
$d(u,y)\in\{1,2\}.$ If $d(u,y)=1,$ then $u\in N_{G}[x]\cap N_{G}[y].$ And if
$d(u,y)=2,$ then the common neighbor of $u$ and $y$ must be adjacent to $x,$
as mentioned above, otherwise $\Delta(x,y)\geq4.$ So, in this case $u\in
N_{G}[N_{G}[x]\cap N_{G}[y]].$
\end{proof}

Notice that the weak $k$-metric dimension of a graph with false twins and
without true twins behaves somehow different to Proposition
\ref{prop:2-met-dim-al}. By taking into account that if $x,y$ are false twins,
then clearly $\Delta_{s}(x,y)>0$ only for $s\in\{x,y\}$, but in this case
$\Delta(x,y)=4$. Thus, also in view of Proposition \ref{prop:no-3-dim-al}, it
follows that $\kappa(G)=\displaystyle\min\left\{  \Delta(x,y)\,:\,x,y\in
V(G)\right\}  =4$. This is summarized in the following corollary.

\begin{corollary}
\label{prop:false-twins} If $G$ has at least two false twins, no true twins
and it does not satisfy the statements of Proposition \ref{prop:no-3-dim-al},
then $G$ is weak $4$-metric dimensional.
\end{corollary}

\paragraph{Basic graph classes.}

We continue this section by finding $\kappa(G)$ for some simple graphs such as
paths, stars, cycles, complete graphs and complete bipartite graphs on $n$
vertices, denoted by $P_{n},$ $S_{n},$ $C_{n},$ $K_{n}$ and $K_{r,s},$
respectively, where $r+s=n.$ We also find $\kappa(G)$ for some specific graph
classes such as trees and grid graphs. In order to do that, recall that
$\kappa(G)=\min\{\Delta(x,y)\,:\,x,y\in V(G)\}$ according to Observation
\ref{Obs_vertexSet}. Therefore, in order to establish the maximum
$k$-dimensionality for the graph classes we consider, we take $S=V(G).$

\begin{proposition}
The following statements hold:

\begin{itemize}
\item[\textrm{(i)}] A complete graph $K_{n},$ for $n\geq2,$ is weak $2$-metric dimensional.

\item[\textrm{(ii)}] A star $S_{n},$ for $n\geq4$, is weak $4$-metric dimensional.

\item[\textrm{(iii)}] A complete bipartite graph $K_{q,r},$ for $q,r\geq2$, is
weak $4$-metric dimensional.

\item[(iv)] A path $P_{n}$, for $n\geq2,$ is weak $n$-metric dimensional.

\item[\textrm{(v)}] A cycle $C_{n}$ is weak $(n-1)$-metric dimensional for odd
$n\geq3,$ and weak $n$-metric dimensional for even $n\geq4.$
\end{itemize}
\end{proposition}

\begin{proof}
A complete graph $K_{n}$ contains a pair of true twins, thus (i) follows from
Proposition \ref{prop:2-met-dim-al}. A star $S_{n}$ and a complete bipartite
graph have a pair of false twins, no pair of true twins, so (ii) and (iii)
follow from Corollary \ref{prop:false-twins}. It remains to prove (iv) and
(v). We consider the set $S=V(G),$ since $\kappa(G)=\min\{\Delta
(x,y)\,:\,x,y\in V(G)\}.$

In order to establish (iv), notice that a pair of neighboring vertices $x,y\in
V(P_{n})$ and any vertex $s\in V(P_{n})$ satisfy $\Delta_{s}(x,y)=1,$ thus
$\Delta(x,y)=\left\vert V(P_{n})\right\vert =n.$ For a pair of vertices
$x,y\in V(P_{n})$ with $d(x,y)\geq2,$ there exists at most one vertex $s\in S$
such that $\Delta_{s}(x,y)=0,$ and if $s\in\{x,y\}$, then $\Delta
_{s}(x,y)=d(x,y)\geq2.$ Since any other vertex $s$ of $S$ satisfies
$\Delta_{s}(x,y)\geq1,$ we conclude that $\Delta(x,y)\geq0+2+2+(n-3)=n+1.$
Thus, $\kappa(P_{n})=\min\{\Delta(x,y)\,:\,x,y\in V(P_{n})\}=n.$

To prove (v), let us first assume that $C_{n}$ is an odd length cycle. For any
pair of neighboring vertices $x,y\in V(C_{n})$, there exists precisely one
$s\in S$ such that $\Delta_{s}(x,y)=0,$ and any other $s\in S$ satisfies
$\Delta_{s}(x,y)=1,$ thus $\Delta(x,y)=n-1.$ For any pair of vertices $x,y\in
V(G)$ at distance at least $2,$ there exists at most one vertex $s\in S$ such
that $\Delta_{s}(x,y)=0.$ On the other hand, for $s\in\{x,y\}$ it holds that
$\Delta_{s}(x,y)=d(x,y)\geq2.$ We conclude that $\Delta(x,y)\geq
0+2+2+(n-3)=n+1.$ Consequently, $\kappa(C_{n})=n-1.$

Let us next assume that $C_{n}$ is an even length cycle. Any pair of
neighboring vertices $x,y\in V(C_{n})$ satisfies $\Delta_{s}(x,y)=1$ for every
$s\in S,$ consequently $\Delta(x,y)=n.$ For any pair of vertices $x,y\in
V(C_{n})$ at distance at least $2$ there exist at most two vertices $s\in
S=V(C_{n})$ such that $\Delta_{s}(x,y)=0.$ In addition, if $s\in\{x,y\},$ then
$\Delta_{s}(x,y)=d(x,y)\geq2.$ Similarly as above, we conclude that
$\Delta(x,y)\geq0+0+2+2+(n-4)=n.$ Consequently, $\kappa(C_{n})=n.$
\end{proof}

\paragraph{Trees.}

In order to establish $\kappa(T)$ for a tree $T,$ we need to introduce several
notions. Given a tree $T$ and a vertex $v\in V(T)$ with $d_{T}(v)\geq3,$ a
\emph{thread} of length $l$ hanging at $v$ is any path $vv_{1}\cdots v_{l}$ in
$T$ such that $d_{T}(v_{i})=2$ for $i\in\lbrack l-1]$ and $d_{T}(v_{l})=1.$ A
vertex $v$ in $T$ with at least $\ell$$(v)\geq2$ threads hanging at it is
called a \emph{root vertex} of $T.$ The number $\ell(v)$ of threads hanging at
a root vertex $v$ is called the \emph{root degree} of $v.$ For a root vertex
$v\in V(T)$ with the root degree $\ell(v)\geq2,$ let $l_{1}(v)\leq
l_{2}(v)\leq\cdots\leq l_{\ell(v)}(v)$ denote the sequence of lengths of
threads hanging at $v.$ When the root vertex is clear from the context, we
will shortly write $\ell$ and $l_{i}.$

Notice that any tree $T\not =P_{n}$ contains at least one root vertex, thus
for such a tree we define
\[
\kappa^{\ast}(T)=\min\{2(l_{1}(v)+l_{2}(v)):v\in R(T)\}
\]
where $R(T)$ is the set of all root vertices of $T.$ When the tree $T$ is
clear from the context, we will shortly write $\kappa^{\ast}.$ For a tree $T$
which has precisely one root vertex $v$ with the root degree $\ell,$ we
introduce the notation $T=S^{\ell}(l_{1},\ldots,l_{\ell})$ where $l_{1}%
\leq\cdots\leq l_{\ell}$ are the lengths of $\ell$ threads hanging at $v.$ One
may consider that $T=S^{\ell}(l_{1},\ldots,l_{\ell})$ is obtained from the
star $S_{\ell+1}$ by edge subdivision.

\begin{lemma}
\label{Lemma_Riste}Let $T$ be a tree, let $x,y\in V(T)$ be two vertices at a
distance $d,$ let $P=u_{0}u_{1}\cdots u_{d}$ be the path connecting $x$ and
$y$ in $T,$ where $x=u_{0}$ and $y=u_{d},$ and let $T_{i}$ be the component of
$T-E(P)$ containing $u_{i}.$ If $S\subseteq V(T),$ then%
\[
\Delta_{S}(x,y)=\sum_{i=0}^{d}\left\vert d-2i\right\vert \left\vert
V(T_{i})\cap S\right\vert .
\]

\end{lemma}

\begin{proof}
Let $s\in V(T_{i})\cap S$ for some $i\in\{0,\ldots,d\}.$ Since the shortest
path from $s$ to both $x$ and $y$ leads through $u_{i},$ it holds that
$\Delta_{s}(x,y)=\left\vert d(s,x)-d(s,y)\right\vert =\left\vert
d(u_{i},x)-d(u_{i},y)\right\vert .$ Further, since $x=u_{0}$ and $y=u_{d},$ it
follows that
\[
\Delta_{s}(x,y)=\left\vert d(u_{i},u_{0})-d(u_{i},u_{d})\right\vert
=\left\vert i-(d-i)\right\vert =\left\vert d-2i\right\vert .
\]
Taking the sum of $\Delta_{s}(x,y),$ where $s$ goes through $V(T_{i})$ and $i$
through $\{0,\ldots,d\}$, yields the result.
\end{proof}

\begin{theorem}
Let $T\not =P_{n}$ be a tree on $n$ vertices. Then $\kappa(T)=\min
\{n,\kappa^{\ast}(T)\}.$ Particularly, if $T\not =S^{3}(l_{1},l_{2},l_{3}),$
then $\kappa(T)=\kappa^{\ast}(T).$
\end{theorem}

\begin{proof}
Recall that $\kappa(T)=\min\{\Delta(x,y)\,:\,x,y\in V(T)\}.$ Let us first
establish that $\kappa(T)\leq\min\{n,\kappa^{\ast}(T)\}.$ For that purpose,
let $v$ be a root vertex of $T$ for which the minimum $\kappa^{\ast}(T)$ is
attained and let $P_{1}$ and $P_{2}$ be two shortest threads hanging at $v.$
Let $x$ and $y$ be vertices of $P_{1}$ and $P_{2},$ respectively, which are
adjacent to $v.$ Notice that $\Delta_{s}(x,y)=2$ for every $s\in V(P_{1})\cup
V(P_{2})\backslash\{v\},$ and $\Delta_{s}(x,y)=0$ for any other $s\in V(T).$
Consequently, we have $\Delta(x,y)=2(l_{1}(v)+l_{2}(v))=\kappa^{\ast}(T).$
Next, we consider any pair of neighboring vertices $x,y\in V(T)$, where
$\Delta_{s}(x,y)=1$ for any $s\in V(T).$ Consequently, we have $\Delta
(x,y)=\left\vert V(T)\right\vert =n.$ This establishes that $\kappa(T)\leq
\min\{n,\kappa^{\ast}(T)\}.$

To prove that $\kappa(T)\geq\min\{n,\kappa^{\ast}(T)\},$ we need to consider
all the remaining pairs of vertices $x,y\in V(T)$ and show that $\Delta
(x,y)\geq\min\{n,\kappa^{\ast}\}$ for all of them. For that purpose, let
$x,y\in V(T)$ be any pair of vertices at distance $d\geq2,$ let $P=u_{0}%
u_{1}\cdots u_{d}$ be the path connecting $x=u_{0}$ and $y=u_{d}$ in $T,$ and
let $T_{i}$ be the component of $T-E(P)$ which contains $u_{i}.$ Lemma
\ref{Lemma_Riste} implies%
\[
\Delta(x,y)=\sum_{i=0}^{d}\left\vert d-2i\right\vert \left\vert V(T_{i}%
)\right\vert .
\]
If $d$ is odd, then $\left\vert d-2i\right\vert \geq1$ for every
$i\in\{0,\ldots,d\},$ so $\Delta(x,y)\geq\sum_{i=0}^{d}\left\vert
V(T_{i})\right\vert =n.$ Hence, let us assume that $d$ is even. Notice that
$\left\vert d-2i\right\vert =0$ if and only if $i=d/2,$ and every integer
$i\not =d/2$ satisfies $\left\vert d-2i\right\vert \geq2.$ Let $i^{\ast}%
\in\{0,\ldots,d\}$ be the smallest integer such that $V(T_{i})$ contains a
root vertex $v$ of $T.$ We may assume that vertices $u_{i}$ of the path $P$
are denoted so that $i^{\ast}\leq d/2.$

If $i^{\ast}<d/2,$ then the set $\bigcup_{i=0}^{i^{\ast}}V(T_{i})$ contains
vertices of at least two threads hanging at a same root vertex of $T.$ Hence,
\[
\sum_{i=0}^{i^{\ast}}\left\vert V(T_{i})\right\vert \geq\min\{l_{1}%
(v)+l_{2}(v):v\in R(T)\},
\]
so we have%
\[
\Delta(x,y)\geq\sum_{i=0}^{i^{\ast}}\left\vert d-2i\right\vert \left\vert
V(T_{i})\right\vert \geq2\min\{l_{1}(v)+l_{2}(v):v\in R(T)\}=\kappa^{\ast
}(T).
\]
If $i^{\ast}=d/2,$ then $u_{d/2}$ is a root vertex and $x$ and $y$ belong to
two distinct threads hanging at $u_{d/2}.$ If we denote $I=\{0,\ldots
,d\}\backslash\{d/2\},$ then the set $\bigcup_{i\in I}V(T_{i})$ contains all
the vertices of the two threads hanging at the root vertex $u_{d/2}$ which
contain $x$ and $y,$ except $u_{d/2}$ itself, hence
\[
\sum_{i\in I}\left\vert V(T_{i})\right\vert \geq\min\{l_{1}(v)+l_{2}(v):v\in
R(T)\}.
\]
and consequently
\[
\Delta(x,y)\geq\sum_{i\in I}\left\vert d-2i\right\vert \left\vert
V(T_{i})\right\vert \geq2\min\{l_{1}(v)+l_{2}(v):v\in R(T)\}=\kappa^{\ast
}(T).
\]
Hence, we have established that $\Delta(x,y)\geq\min\{n,\kappa^{\ast}(T)\}$
for all $x,y\in V(T),$ which implies $\kappa(T)\geq\min\{n,\kappa^{\ast}\}.$
We conclude that $\kappa(T)=\min\{n,\kappa^{\ast}(T)\}$ for every
$T\not =P_{n}.$

It remains to prove that $\kappa(T)=\kappa^{\ast}(T)$ for $T\not =P_{n}$ and
$T\not =S^{3}(l_{1},l_{2},l_{3}).$ It is sufficient to establish that for such
trees $T$ the inequality $n\geq\kappa^{\ast}(T)$ holds. Assume first that $T$
has precisely one root vertex. Then $T=S^{\ell}(l_{1},\ldots,l_{\ell})$ for
$\ell\geq4,$ in which case $n\geq l_{1}+\cdots+l_{\ell}+1\geq2(l_{1}%
+l_{2})+1>\kappa^{\ast}(T).$ Assume next that $T$ has at least two root
vertices, say $v$ and $w.$ Then $n\geq l_{1}(v)+l_{2}(v)+l_{1}(w)+l_{2}%
(w)+2>\kappa^{\ast}(T),$ and we are done.
\end{proof}

\begin{figure}[h]
\begin{center}%
\begin{tabular}
[c]{cc}%
\begin{tabular}
[t]{ll}%
a) & \raisebox{-0.9\height}{\includegraphics[scale=0.6]{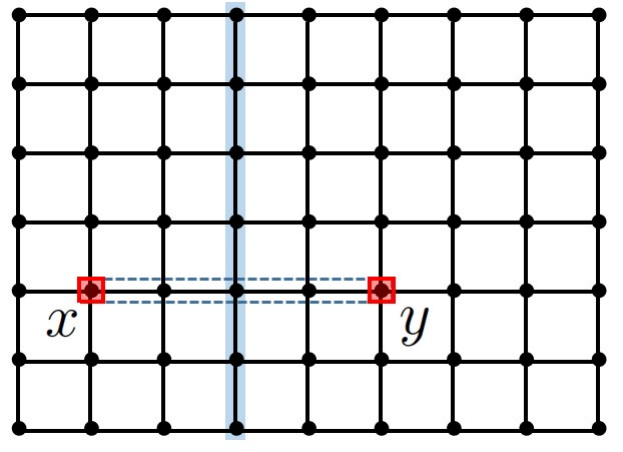}}
\end{tabular}
&
\begin{tabular}
[t]{ll}%
b) & \raisebox{-0.9\height}{\includegraphics[scale=0.6]{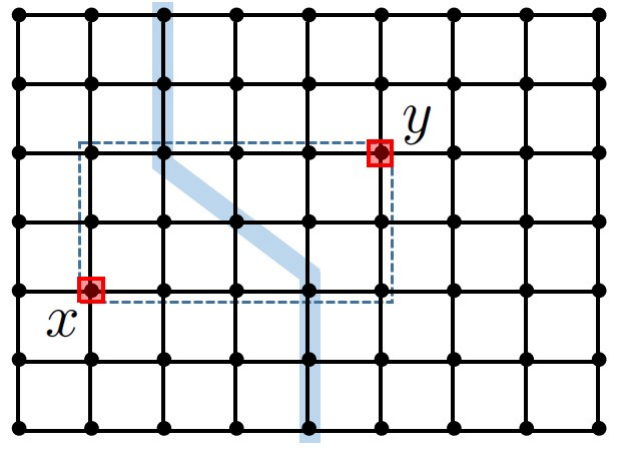}}
\end{tabular}
\medskip\\
\multicolumn{2}{c}{%
\begin{tabular}
[t]{ll}%
c) & \raisebox{-0.9\height}{\includegraphics[scale=0.6]{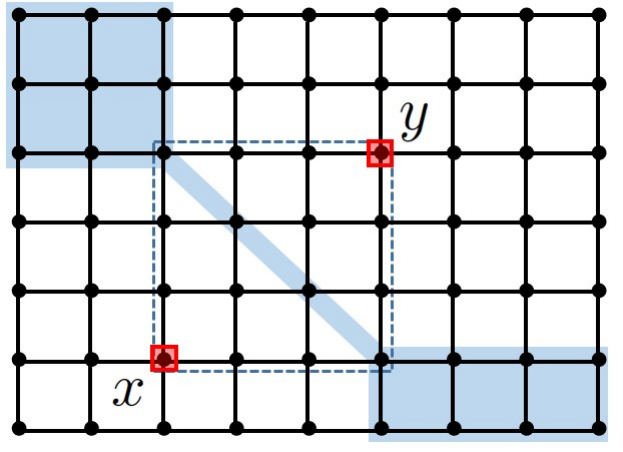}}
\end{tabular}
}%
\end{tabular}
\end{center}
\caption{Each figure shows the grid graph $G=P_{9}\Box P_{7}$, and a pair of
vertices $x,$ $y$ in it such that $d(x,y)$ is even. The set of vertices
$S_{0}(x,y)$ which contains all vertices $s\in V(G)$ such that $\Delta
_{s}(x,y)=0$ is shaded for: a) $\Delta j=0,$ b) $0<\Delta j<\Delta i,$ c)
$0<\Delta j=\Delta i.$}%
\label{Fig_S0}%
\end{figure}

\paragraph{Grid graphs.}

For a pair of graphs $G_{1}$ and $G_{2}$ with set of vertices $V(G_{1}%
)=\{u_{1},\ldots,u_{n_{1}}\}$ and $V(G_{2})=\{v_{1},\ldots,v_{n_{2}}\},$
respectively, the \emph{Cartesian product} $G_{1}\Box G_{2}$ is defined with
$V(G_{1}\Box G_{2})=V(G_{1})\times V(G_{2})$ and a pair of vertices
$(u_{i_{1}},v_{j_{1}})$ and $(u_{i_{2}},v_{j_{2}})$ is connected in $G_{1}\Box
G_{2}$ by an edge if and only if $u_{i_{1}}=u_{i_{2}}$ and $v_{j_{1}}v_{j_{2}%
}\in E(G_{2}),$ or $v_{j_{1}}=v_{j_{2}}$ and $u_{i_{1}}u_{i_{2}}\in E(G_{1}).$
In this paper we will consider a Cartesian product of paths, which is also
called a \emph{grid graph}. Let us now consider $\kappa(G)$ for a grid graph
$G.$ First, notice that the following auxiliary lemma holds.

\begin{lemma}
\label{Lemma_riste2}Let $G$ be a bipartite graph and let $x,y$ be a pair of
vertices of $G$. If $x$ and $y$ belong to distinct bipartition sets of $G,$
then $\Delta_{s}(x,y)\geq1$ for every $s\in V(G).$
\end{lemma}

\begin{proof}
Since $x$ and $y$ belong to different bipartition sets of $G,$ it follows that
for every $s\in V(G)$ the numbers $d(x,s)$ and $d(y,s)$ are of distinct
parity. Consequently, $\Delta_{s}(x,y)=\left\vert d(x,s)-d(y,s)\right\vert
\geq1.$
\end{proof}

Notice that $G=P_{q}\Box P_{r}$ is a bipartite graph, and a pair of vertices
$x,y$ of $G$ belongs to distinct bipartition sets if and only if $d(x,y)$ is
odd, so the previous lemma applies to such pairs. Let us establish the weak
dimensionality of a grid graph in the next theorem.

\begin{theorem}
\label{Prop_grid_kapa}Any grid graph $G=P_{q}\Box P_{r},$ where $q,r\geq2,$ is
weak $(2q+2r-4)$-metric dimensional.
\end{theorem}

\begin{proof}
Let $P_{q}=u_{1}u_{2}\cdots u_{q}$ and $P_{r}=v_{1}v_{2}\cdots v_{r}.$
Further, let $V(G)=\{(u_{i},v_{j}):u_{i}\in V(P_{q}),v_{j}\in V(P_{r})\}.$
Since $\kappa(G)=\min\{\Delta(x,y)\,:\,x,y\in V(G)\},$ we consider the set
$S=V(G)$.

Let us first establish that $\kappa(G)\leq2q+2r-4.$ To do that, let us
consider the vertices $x=(u_{1},v_{2})$ and $y=(u_{2},v_{1}).$ Notice that
$\Delta_{s}(x,y)=2$ for any $s\in\{(u_{i},v_{1}):i\in\{2,\ldots,q\}\}\cup
\{(u_{1},v_{j}):j\in\{2,\ldots,r\}\},$ and each other $s\in S$ satisfies
$\Delta_{s}(x,y)=0.$ Hence, $\Delta(x,y)=2(q-1)+2(r-1)=2q+2r-4,$ which implies
$\kappa(G)\leq2q+2r-4.$

Let us now prove that $\kappa(G)\geq2q+2r-4.$ We need to show that
$\Delta(x,y)\geq2q+2r-4$ for every pair $x,y\in V(G).$ If $d(x,y)$ is odd,
then Lemma \ref{Lemma_riste2} implies $\Delta_{s}(x,y)\geq1$ for every $s\in
V(G),$ hence $\Delta(x,y)\geq\left\vert V(G)\right\vert =qr\geq2q+2r-4$ for
$q,r\geq2,$ so the claim holds.\ Assume therefore that $d(x,y)$ is even, and
let $x=(u_{i_{1}},v_{j_{1}})$, $y=(u_{i_{2}},v_{j_{2}}).$ We denote $\Delta
i=i_{2}-i_{1}$ and $\Delta j=j_{2}-j_{1},$ and we may assume that $0\leq\Delta
j\leq\Delta i.$ Denote by $S_{0}(x,y)$ the set of all vertices $s$ of the grid
such that $\Delta_{s}(x,y)=0.$ The set $S_{0}(x,y)$ is illustrated by Figure
\ref{Fig_S0}. Since $d(x,y)$ is even, every $s\in V(G)\backslash S_{0}(x,y)$
satisfies $\Delta_{s}(x,y)\geq2$.

Let
\[
R=\{(u_{i},v_{j}):i\in\{1,q\}\text{ or }j\in\{1,r\}\}
\]
be the set of vertices on the border of $G.$ We claim that at least half of
the vertices of $R$ are not contained in $S_{0}(x,y).$ Assume first that
$j_{1}=j_{2},$ i.e. $\Delta j=0.$ Then $S_{0}(x,y)=\{(u_{i},v_{j}%
):i=i_{1}+\Delta i/2,j\in\lbrack r]\}$ as shown in Figure \ref{Fig_S0}.a), so
$\left\vert S_{0}(x,y)\cap R\right\vert =2.$ Since $R$ contains at least $4$
vertices, the claim holds. Assume next that $j_{1}\not =j_{2},$ i.e. $\Delta
j=j_{2}-j_{1}>0.$ Notice that $s=(u_{i},v_{j})\in S_{0}(x,y)$ implies
$d(x,s)=d(y,s),$ and so%
\[
\left\vert i_{1}-i\right\vert +\left\vert j_{1}-j\right\vert =\left\vert
i_{2}-i\right\vert +\left\vert j_{2}-j\right\vert .
\]
Let us consider the pair of vertices\thinspace$(u_{i},v_{1})$ and
$(u_{i},v_{r})$ from $R.$ Assuming that both $(u_{i},v_{1})$ and $(u_{i}%
,v_{r})$ belong to $S_{0}(x,y),$ from the above equality we obtain $\left\vert
i_{1}-i\right\vert =\left\vert i_{2}-i\right\vert +\Delta j$ and $\left\vert
i_{1}-i\right\vert +\Delta j=\left\vert i_{2}-i\right\vert $, a contradiction.
Hence, at most one of these two vertices can belong to $S_{0}(x,y).$ This is
illustrated by Figures \ref{Fig_S0}.b) and \ref{Fig_S0}.c). Similarly, at most
one vertex from the pair $(u_{1},v_{j}),(u_{q},v_{j})\in R$ can belong to
$S_{0}(x,y),$ so the claim again holds.

By the above, we have that at least half of vertices of $R$ do not belong to
$S_{0}(x,y).$ Since $\left\vert R\right\vert =2q+2r-4,$ we conclude
$\Delta(x,y)\geq2\cdot\left\vert R\right\vert /2=2q+2r-4,$ and we are done.
\end{proof}

\section{Weak $k$-metric dimension for some graph classes}

This section is focused on computing the value of the weak $k$-metric
dimension of some graphs. To this end, we use the suitable values for $k$
already proved to be satisfied in Section \ref{sec:kappa}. First, the
following lemma is also useful for establishing the weak $k$-metric dimension
of graphs.

\begin{proposition}
\label{Lemma_auxiliary}If $G$ is a connected graph with $n$ vertices, then
$k\leq\operatorname*{wdim}_{k}(G)\leq n.$
\end{proposition}

\begin{proof}
Let $x,y\in V(G)$ be a pair of neighboring vertices of $G.$ Then
\[
\Delta_{s}(x,y)=\left\vert d(x,s)-d(y,s)\right\vert \leq1
\]
for every $s\in V(G).$ So, since $\Delta_{S}(x,y)=\sum_{s\in S}\left\vert
d(x,s)-d(y,s)\right\vert \geq k$ needs to be satisfied, the set $S$ must
contain at least $k$ vertices, which implies $\operatorname*{wdim}_{k}(G)\geq
k$. The claim $\operatorname*{wdim}_{k}(G)\leq n$ is the direct consequence of
the fact that $S$ can contain at most $n$ vertices.
\end{proof}

The obvious consequence of the lemma above is that $\kappa(G)\leq n$ for any
graph $G$ on $n$ vertices. Let us now establish the weak $k$-metric dimension
of some simple graphs such as paths, stars, cycles, complete graphs and
complete bipartite graphs.

\begin{proposition}
The following statements hold:

\begin{itemize}
\item[\textrm{(i)}] If $K_{n}$ is a complete graph with $n\geq2$ vertices,
then $\operatorname*{wdim}_{k}(K_{n})=n-1$ for $k=1,$ and
$\operatorname*{wdim}_{k}(K_{n})=n$ for $k=2.$

\item[\textrm{(ii)}] If $S_{n}$ is a star with $n\geq5$ vertices, then
$\operatorname*{wdim}_{k}(S_{n})=n-2$ for $k\in\{1,2\},$ and
$\operatorname*{wdim}_{k}(S_{n})=n-1$ for $k\in\{3,4\}.$

\item[\textrm{(iii)}] If $K_{q,r}$ is a complete bipartite graph with
$q,r\geq2,$ then $\operatorname*{wdim}_{k}(K_{q,r})=q+r-2$ for $k\in\{1,2\},$
and $\operatorname*{wdim}_{k}(K_{q,r})=q+r$ for $k\in\{3,4\}.$

\item[(iv)] If $P_{n}$ is a path with $n\geq2$ vertices, then
$\operatorname*{wdim}_{k}(P_{n})=k$ for $k\in\{1,\ldots,n\}.$

\item[\textrm{(v)}] If $C_{n}$ is a cycle with $n\geq5$ vertices, then
$\operatorname*{wdim}_{1}(C_{n})=2,$ and for $k\in\{2,\ldots,n\}$ it holds
that $\operatorname*{wdim}_{k}(C_{n})=k+1$ if $n$ is odd, and
$\operatorname*{wdim}_{k}(C_{n})=k$ if $n$ is even.
\end{itemize}
\end{proposition}

\begin{proof}
$(i)$ Denote vertices of $K_{n}$ by $u_{1},\ldots,u_{n}.$ Let us first
establish that $\operatorname*{wdim}_{k}(K_{n})\geq n-1$ for every
$k\in\{1,2\}.$ For that purpose, we need to show that a set $S\subseteq
V(K_{n})$ such that $\left\vert S\right\vert \leq n-2$ cannot be a weak
$k$-resolving set for any $k.$ Since $\left\vert S\right\vert \leq n-2,$ it
follows that $S$ does not contain at least two vertices of $K_{n},$ say
$u_{1}$ and $u_{2}.$ Then $\Delta_{S}(u_{1},u_{2})=0,$ so $S$ is not a weak
$k$-resolving set. Let us now consider a set $S\subseteq V(K_{n})$ with
$\left\vert S\right\vert =n-1.$ We need to establish that at least one such
set is a weak $1$-resolving set, hence $\operatorname*{wdim}_{1}(K_{n})=n-1,$
and that none of them is a weak $2$-resolving set. Since $\left\vert
S\right\vert =n-1,$ the set $S$ does not contain at least one vertex from
$K_{n},$ say $u_{1}\not \in S.$ It follows that $\Delta_{S}(u_{1},u_{j})=1$
for every $j\not =1,$ and $\Delta_{S}(u_{i},u_{j})=2$ for every $2\leq i<j\leq
n.$ Thus, since $\Delta_{S}(u_{i},u_{j})\geq1$ for every $1\leq i<j\leq n,$ it
follows that $S$ is a weak $1$-resolving set and $\operatorname*{wdim}%
_{1}(K_{n})=n-1.$ On the other hand, since $\Delta_{S}(u_{1},u_{j})=1$ for
every $j\not =1$, we conclude that $S$ is not a weak $2$-resolving set.
Finally, in order to establish that $\operatorname*{wdim}_{2}(K_{n})=n,$ it is
sufficient to show that for $S=V(K_{n})$ it holds that $\Delta_{S}(u_{i}%
,u_{j})=2$ for every $1\leq i<j\leq n,$ and this is obvious.$\bigskip$

$(ii)$ Denote the central vertex of $S_{n}$ by $u$ and its neighbors by
$u_{1},\ldots,u_{n-1}.$ Let us first establish that a set $S\subseteq
V(S_{n})$ with $\left\vert S\right\vert \leq n-3$ cannot be a weak
$k$-resolving set for any $k\geq1.$ Since $\left\vert S\right\vert \leq n-3,$
it follows that $S$ does not contain at least two leaves, say $u_{1}$ and
$u_{2}.$ Then $\Delta(u_{1},u_{2})=0,$ so $S$ is not a weak $k$-resolving set.

We next consider a set $S\subseteq V(G)$ with $\left\vert S\right\vert =n-2.$
We must establish that at least one such set is a weak $k$-resolving set for
$k\in\{1,2\},$ hence $\operatorname*{wdim}_{1}(S_{n})=\operatorname*{wdim}%
_{2}(S_{n})=n-2,$ and that no such set is a weak $k$-resolving set for
$k\in\{3,4\}.$ Since $\left\vert S\right\vert =n-2,$ the set $S$ does not
contain precisely two vertices of $S_{n}.$ Assume first that $S$ does not
contain two leaves of $S,$ say $V(S_{n})\backslash S=\{u_{1},u_{2}\},$ then
$\Delta_{S}(u_{1},u_{2})=0$ and consequently $S$ is not a weak $k$-resolving
set for any $k\geq1.$ Assume next that $S$ does not contain a central vertex
$u$ and one of the leaves, say $V(S_{n})\backslash S=\{u,u_{1}\}.$ First, for
every integer $1\leq j\leq n-1,$ we have $\Delta_{S}(u,u_{j})=n-2\geq5-2=3.$
Next, for any integer $2\leq j\leq n-1$, we have $\Delta_{S}(u_{1},u_{j})=2.$
Finally, for any pair of integers $2\leq i<j\leq n-1,$ we have $\Delta
_{S}(u_{i},u_{j})=4.$ Since $\Delta_{S}(x,y)\geq2$ for any $x,y\in V(S_{n}),$
we conclude that $S$ is a weak $k$-resolving set for $k\in\{1,2\}.$ On the
other hand, since $\Delta_{S}(u_{1},u_{j})=2,$ we conclude that $S$ is not a
weak $k$-resolving set for $k\geq3.$

Finally, in order to establish that $\operatorname*{wdim}_{3}(S_{n}%
)=\operatorname*{wdim}_{4}(S_{n})=n-1,$ it is sufficient to find one set
$S\subseteq V(S_{n})$ with $\left\vert S\right\vert =n-1$ which is a weak
$k$-resolving set for $k\in\{3,4\}.$ For that purpose, let us consider the set
$S=\{u_{1},\ldots,u_{n-1}\}.$ Notice that for every pair of integers $1\leq
i<j\leq n-1,$ it holds that $\Delta_{S}(u_{i},u_{j})=4.$ Also, for every
integer $1\leq i\leq n-1,$ it holds that $\Delta_{S}(u_{i},u)=\left\vert
S\right\vert =n-1\geq5-1=4.$ Thus, $S$ is a weak $k$-resolving set for $k=3$
and $k=4.$ $\bigskip$

$(iii)$ Let $A$ and $B$ be the bipartition sets of $V(K_{q,r})$, such that
$\left\vert A\right\vert =q$ and $\left\vert B\right\vert =r.$ Assume further
that vertices of $A$ and $B$ are denoted by $A=\{a_{1},\ldots,a_{q}\}$ and
$B=\{b_{1},\ldots,b_{r}\}.$ Let us first prove that $\operatorname*{wdim}%
_{k}(K_{q,r})\geq q+r-2$ for any $k\geq1.$ In order to do so, let $S\subseteq
V(K_{q,r})$ such that $\left\vert S\right\vert \leq q+r-3,$ and we have to
show that $S$ cannot be a weak $k$-resolving set for any $k.$ Notice that $S$
does not contain at least two vertices from at least one bipartition set, say
$a_{1}$ and $a_{2}$ from $A.$ Since $a_{1}$ and $a_{2}$ are false twins, we
have $\Delta_{S}(a_{1},a_{2})=0$, so $S$ is not a weak $k$-resolving set.

Next, let us prove that $\operatorname*{wdim}_{1}(K_{q,r}%
)=\operatorname*{wdim}_{2}(K_{q,r})=q+r-2.$ It is sufficient to find a weak
$k$-resolving set $S$ of $K_{p,q}$, for $k\in\{1,2\},$ such that $\left\vert
S\right\vert =q+r-2.$ For that purpose, let $S=\{a_{i},b_{j}:i\in\lbrack
q-1],j\in\lbrack r-1]\},$ and we have to show that $\Delta_{S}(x,y)\geq2$ for
every $x,y\in V(K_{q,r}).$ Notice that $\Delta_{s}(a_{i},b_{j})\geq1$ for any
$s\in V(K_{q,r}),$ thus $\Delta_{S}(a_{i},b_{j})\geq\left\vert S\right\vert
=q+r-2\geq2.$ Notice further that $\Delta_{S}(a_{i},a_{j})=4$ if
$q\not \in \{i,j\},$ otherwise $\Delta_{S}(a_{i},a_{j})=2.$ Similarly,
$\Delta_{S}(b_{i},b_{j})=4$ if $r\not \in \{i,j\},$ otherwise $\Delta
_{S}(b_{i},b_{j})=2.$ We conclude that $S$ is a weak $k$-resolving set for
$k\in\{1,2\}.$

Finally, let us show that $\operatorname*{wdim}_{3}(K_{q,r}%
)=\operatorname*{wdim}_{3}(K_{q,r})=q+r.$ In order to do so, let us first
establish that any set $S\subseteq V(K_{q,r})$ with $\left\vert S\right\vert
\leq q+r-1$ cannot be a weak $k$-resolving set for $k\geq3.$ Notice that such
a set $S$ does not contain at least one vertex from $V(K_{p,q})$, say
$a_{1}\not \in S.$ Then, $\Delta_{s}(a_{1},a_{2})=2$ for $s=a_{2},$ otherwise
$\Delta_{s}(a_{1},a_{2})=0.$ Consequently, we have $\Delta_{S}(a_{1}%
,a_{2})\leq2,$ so $S$ is not a weak $k$-resolving set for $k\geq3.$ Next, we
have to find at least one weak $k$-resolving set $S$ with $\left\vert
S\right\vert =q+r$ for $k\in\{3,4\}.$ For that purpose, let us consider
$S=V(K_{q,r}).$ Notice that $\Delta_{S}(a_{i},a_{j})=4,$ $\Delta_{S}%
(b_{i},b_{j})=4,$ and $\Delta_{S}(a_{i},b_{j})=\left\vert S\right\vert
=q+r\geq4.$ Thus, $S$ is a weak $k$-resolving set for $k\in\{3,4\}$ and we are
done.$\bigskip$

$(iv)$ Assume the notation $P_{n}=u_{1}u_{2}\cdots u_{n}.$ Proposition
\ref{Lemma_auxiliary} implies $\operatorname*{wdim}_{k}(P_{n})\geq k$ for
every feasible $k.$ So, let us prove that $\operatorname*{wdim}_{k}(P_{n})\leq
k$, by showing that any $S=\{u_{1},\ldots,u_{k}\}\subseteq V(P_{n})$ is a weak
$k$-resolving set. Notice that for any pair $x,y\in V(P_{n})$ there exists at
most one vertex $s_{0}\in S$ such that $\Delta_{s_{0}}(x,y)=0.$ If such a
vertex $s_{0}\in S$ does not exist, then $\Delta_{s}(x,y)\geq1$ for every
$s\in S,$ so $\Delta_{S}(x,y)\geq\left\vert S\right\vert =k.$ On the other
hand, if such a vertex $s_{0}\in S$ does exist, this implies that $x$ and $y$
are on an even distance in $P_{n},$ hence $d(x,y)\geq2$ and $\Delta
_{s}(x,y)\geq2$ for every $s\in S\backslash\{s_{0}\}.$ Also, the existence of
such a vertex $s_{0}$ implies $k\geq2,$ since for $k=1$ the set $S$ consists
only of an end-vertex of $P_{n}.$ We conclude that $\Delta_{S}(x,y)\geq
0+2(k-1)\geq k$ for $k\geq2.$ Therefore, we have established that $\Delta
_{S}(x,y)\geq k$ for every $x,y\in V(P_{n}),$ which implies $S$ is a weak
$k$-resolving set, so $\operatorname*{wdim}_{k}(P_{n})\leq\left\vert
S\right\vert =k.$

$\bigskip$

$(v)$ For $k=1$ the weak $k$-dimension $\operatorname*{wdim}_{k}(G)$ is the
classical (vertex) metric dimension, so it is already known in the literature
that $\operatorname*{wdim}_{1}(C_{n})=2.$ Hence, assume that $k\geq2$ and
denote $C_{n}=u_{1}u_{2}\cdots u_{n}u_{1}.$ Vertices $x,y$ of an even length
cycle $C_{n}$ are \emph{antipodal}, if $d(x,y)=n/2.$ On an odd length cycle
$C_{n},$ a pair of vertices $y$ and $y^{\prime}$ is \emph{antipodal} to a
vertex $x$ if $d(x,y)=d(x,y^{\prime})=\left\lfloor n/2\right\rfloor .$ Notice
that in this case vertices $y$ and $y^{\prime}$ must be neighbors.

Proposition \ref{Lemma_auxiliary} implies $\operatorname*{wdim}_{k}(C_{n})\geq
k.$ Also, for any pair of neighboring vertices $x,y\in V(C_{n})$ and any
vertex $s\in V(C_{n})$ it holds that $\Delta_{s}(x,y)\leq1,$ where $\Delta
_{s}(x,y)=0$ if and only if $x$ and $y$ are the antipodal pair of $s$ on an
odd length cycle $C_{n}.$

Let us first consider the case of an odd length cycle $C_{n}.$ To establish
that $\operatorname*{wdim}_{k}(C_{n})\geq k+1,$ we have to show that any set
$S\subseteq V(G)$ with $\left\vert S\right\vert =k$ cannot be a weak
$k$-resolving set. For that purpose, let $s_{0}\in S$ be any vertex of $S.$ If
$x,y$ form an antipodal pair of vertices of the vertex $s_{0},$ then
$\Delta_{s_{0}}(x,y)=0.$ Since $x$ and $y$ are neighbors, we have that
$\Delta_{s}(x,y)=1$ for every $s\in S\backslash\{s_{0}\}$. We conclude that
$\Delta_{S}(x,y)=\left\vert S\right\vert -1=k-1,$ so $S$ is not a weak
$k$-resolving set and $\operatorname*{wdim}_{k}(C_{n})\geq k+1.$ To establish
that $\operatorname*{wdim}_{k}(C_{n})\leq k+1,$ it is sufficient to find a set
$S\subseteq V(C_{n})$ with $\left\vert S\right\vert =k+1$ which is a weak
$k$-resolving set. For that purpose, consider the set $S=\{u_{1}%
,\ldots,u_{k+1}\}\subseteq V(C_{n}).$ Let $x,y$ be any pair of vertices of
$C_{n}$. Observe that there exists at most one vertex $s_{0}\in S$ such that
$\Delta_{s_{0}}(x,y)=0,$ and for every $s\in S\backslash\{s_{0}\}$ we have
that $\Delta_{s}(x,y)\geq1.$ Thus $\Delta_{S}(x,y)\geq k,$ so $S$ is a weak
$k$-resolving set, which implies $\operatorname*{wdim}_{k}(C_{n}%
)\leq\left\vert S\right\vert =k+1.$

Let us now consider the case of an even length cycle $C_{n}.$ Proposition
\ref{Lemma_auxiliary} implies $\operatorname*{wdim}_{k}(C_{n})\geq k$, so it
remains to establish that $\operatorname*{wdim}_{k}(C_{n})\leq k.$ It is
sufficient to find a weak $k$-resolving set $S$ in $C_{n}$ with $\left\vert
S\right\vert =k.$ For that purpose, consider the set $S=\{u_{1},\ldots
,u_{k}\}.$ Observe that for a pair of vertices $x,y\in V(C_{n})$ there are at
most two vertices $s\in S$ such that $\Delta_{s}(x,y)=0.$ Now, if there are no
vertices in $S$ such that $\Delta_{s}(x,y)=0,$ then $\Delta_{s}(x,y)\geq1$ for
every $s\in S,$ so $\Delta_{S}(x,y)\geq\left\vert S\right\vert =k.$ If there
is precisely one vertex $s_{0}\in S$ such that $\Delta_{s_{0}}(x,y)=0,$ then
there exists a neighbor $s_{1}\in S$ of $s_{0}$ such that $\Delta_{s_{1}%
}(x,y)=2,$ and for all other vertices $s\in S$ we have that $\Delta
_{s}(x,y)\geq1.$ Thus, $\Delta_{S}(x,y)\geq0+2+(k-2)=k.$ Finally, if there are
two vertices $s_{0},s_{0}^{a}\in S$ such that $\Delta_{s_{0}}(x,y)=\Delta
_{s_{0}^{a}}(x,y)=0,$ then $s_{0}$ and $s_{0}^{a}$ must be an antipodal pair
on $C_{n}.$ Therefore, $n\not =4$ implies that there exists a neighbor
$s_{1}\in S$ of $s_{0}$ and a neighbor $s_{1}^{a}\in S$ of $s_{0}^{a}$ such
that $s_{1}\not =s_{1}^{a}$, $\Delta_{s_{1}}(x,y)\geq2$ and $\Delta_{s_{2}%
}(x,y)\geq2.$ In addition, any other vertex $s\in S$ satisfies $\Delta
_{s}(x,y)\geq1,$ so we conclude $\Delta_{S}(x,y)\geq0+2+0+2+(k-4)=k.$ We have
shown that $\Delta_{S}(x,y)\geq k$ for all pairs $x,y\in V(C_{n}),$ so $S$ is
a weak $k$-resolving set, which implies $\operatorname*{wdim}_{k}(C_{n}%
)\leq\left\vert S\right\vert =k,$ and we are done.
\end{proof}

In view of Proposition \ref{Lemma_auxiliary}, it would be interesting to
characterize all the graphs attaining the equality in such bounds.

\subsection{Trees}

Next, we wish to determine the weak $k$-metric dimension of a tree $T$ for
every $k$ such that the weak $k$-metric dimension is defined. We need to
consider the tree $T=S^{3}(l_{1},l_{2},l_{3})$ separately, since the weak
$k$-metric dimension of such trees behaves differently. \begin{figure}[ptbh]
\begin{center}%
\begin{tabular}
[c]{ll}%
a) \raisebox{-0.9\height}{\includegraphics[scale=0.6]{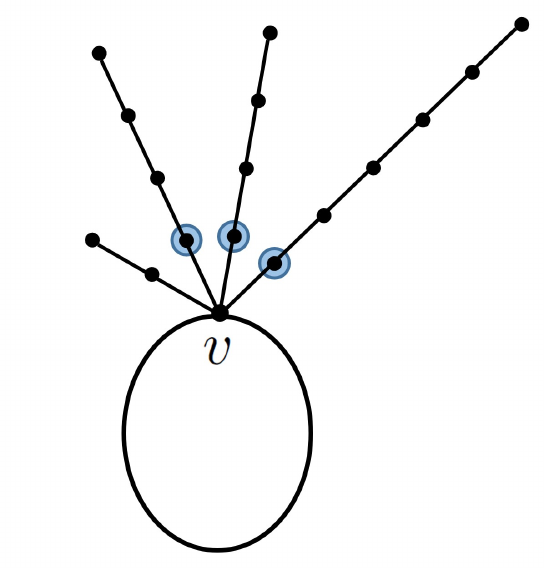}} & b)
\raisebox{-0.9\height}{\includegraphics[scale=0.6]{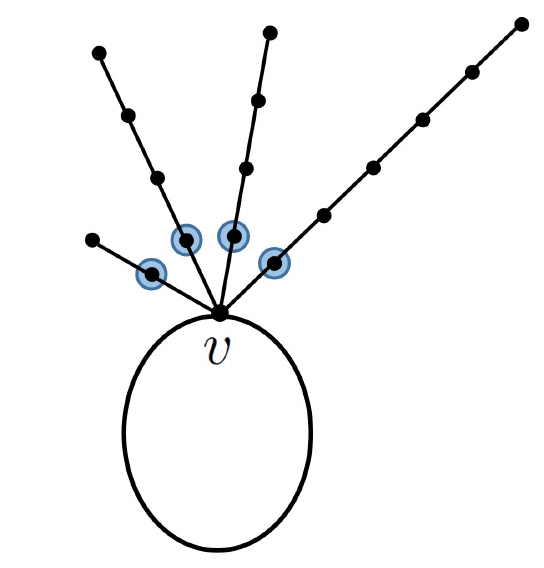}}\\
c) \raisebox{-0.9\height}{\includegraphics[scale=0.6]{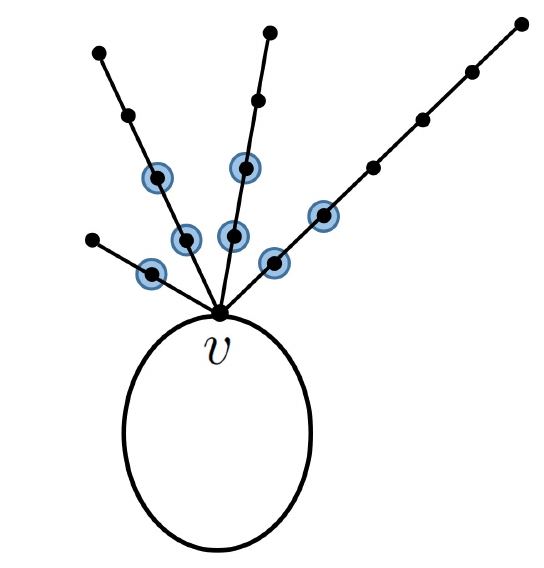}} & d)
\raisebox{-0.9\height}{\includegraphics[scale=0.6]{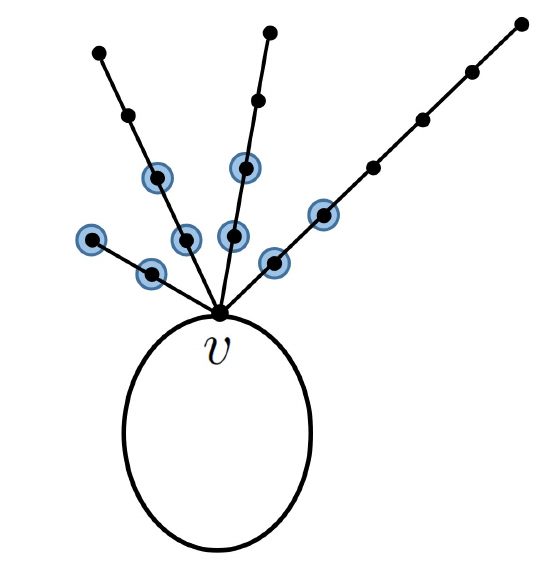}}\\
e) \raisebox{-0.9\height}{\includegraphics[scale=0.6]{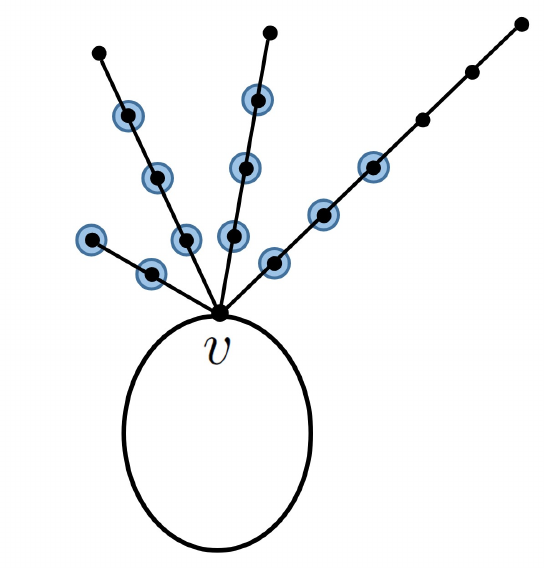}} & f)
\raisebox{-0.9\height}{\includegraphics[scale=0.6]{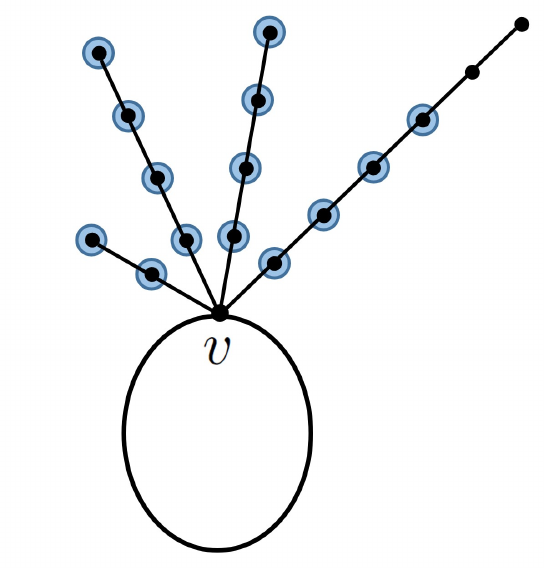}}
\end{tabular}
\end{center}
\caption{Each figure shows a root vertex $v$ with $\ell(v)=4$ and the lengths
of threads $(l_{1},l_{2},l_{3},l_{4})=(2,4,4,6).$ The set $S_{k}(v)$ is marked
in a figure for: a) $k\in\{1,2\},$ b) $k\in\{3,4\},$ c) $k\in\{5,6\}$, d)
$k\in\{7,8\},$ e) $k\in\{9,10\},$ f) $k\in\{11,12\}.$ For higher values of $k$
the weak $k$-metric dimension is not defined.}%
\label{Fig_trees}%
\end{figure}

\begin{proposition}
If $T=S^{3}(l_{1},l_{2},l_{3})$ is a tree on $n$ vertices, then
$\operatorname*{wdim}\nolimits_{1}(T)=2$ and $\operatorname*{wdim}%
\nolimits_{k}(T)=k$ for $2\leq k\leq\kappa(T)=\min\{n,\kappa^{\ast}(T)\}.$
\end{proposition}

\begin{proof}
For $k=1,$ the weak $k$-metric dimension equals the classical (vertex) metric
dimension, so the result $\operatorname*{wdim}_{1}(T)=\dim(T)=2$ is already
known in the literature. Let us assume that $k\geq2.$ Proposition
\ref{Lemma_auxiliary} implies that $\operatorname*{wdim}_{k}(T)\geq k,$ hence
to prove the claim it is sufficient to establish that $\operatorname*{wdim}%
_{k}(T)\leq k.$ Let $v$ be the only root vertex of $T,$ and let $P^{(j)}%
=vv_{1}^{(j)}\cdots v_{l_{j}}^{(j)}$ be a thread hanging at $v$ for
$j\in\lbrack3].$ We define the function $f:V(T)\rightarrow\lbrack n]$ with
$f(v)=n$ and $f(v_{i_{1}}^{(j_{1})})<f(v_{i_{2}}^{(j_{2})})$ if and only if
$i_{1}<i_{2}$ or $i_{1}=i_{2}$ and $j_{1}<j_{2}.$ Notice that $f$ is a
bijection. For every $2\leq k\leq\kappa(T),$ we define a set $S_{k}=\{s\in
V(T):f(s)\leq k\}.$ Since $f$ is a bijection, we have $\left\vert
S_{k}\right\vert =k.$ To prove that $\operatorname*{wdim}_{k}(T)\leq k,$ it is
sufficient to show that $S_{k}$ is a weak $k$-resolving set.

Let $x,y\in V(T)$ be a pair of vertices of $T.$ If $d(x,y)$ is odd, then
$\Delta_{s}(x,y)\geq1$ for every $s\in V(T),$ so $\Delta_{S_{k}}%
(x,y)\geq\left\vert S_{k}\right\vert =k.$ So, let us assume that $d(x,y)$ is
even. If $d(x,v)\not =d(y,v),$ there exists at most one $s_{0}\in S_{k}$ such
that $\Delta_{s_{0}}(x,y)=0,$ and all other $s\in S_{k}$ satisfy $\Delta
_{s}(x,y)\geq2.$ Consequently, $\Delta_{S_{k}}(x,y)\geq2(k-1)\geq k$ for
$k\geq2.$

If $d(x,v)=d(y,v),$ then $x$ and $y$ belong to two distinct threads hanging at
the root $v,$ denote them by $P^{x}$ and $P^{y}$ respectively. Any $s\in
S_{k}$ which belongs to $P^{x}$ or $P^{y}$ satisfies $\Delta_{s}(x,y)\geq2,$
and any other $s\in S_{k}$ satisfies $\Delta_{s}(x,y)=0.$ We deduce that
\[
\Delta_{S_{k}}(x,y)\geq2\left\vert (V(P^{x})\cup V(P^{y}))\cap S_{k}%
\right\vert .
\]
Assume first that $k\leq3l_{1}.$ For $k=3p+q$, where $p\geq0$ is an integer
and $q\in\{0,1,2\},$ we have%
\[
\Delta_{S_{k}}(x,y)\geq\left\{
\begin{array}
[c]{cc}%
2(\left\lfloor k/3\right\rfloor +\left\lfloor k/3\right\rfloor ) & \text{if
}q\not =2,\\
2(\left\lceil k/3\right\rceil +\left\lfloor k/3\right\rfloor ) & \text{if
}q=2.
\end{array}
\right.
\]
Hence, we promptly obtain $\Delta_{S_{k}}(x,y)\geq k.$ Assume next that
$3l_{1}<k\leq\min\{2l_{1}+2l_{2},n\},$ then%
\begin{align*}
\Delta_{S_{k}}(x,y)  &  \geq2(2l_{1}+\left\lfloor (k-3l_{1})/2\right\rfloor
)\geq2(2l_{1}+(k-3l_{1}-1)/2)\\
&  =k+l_{1}-1\geq k.
\end{align*}
Thus, we have established that any pair $x,y\in V(T)$ satisfies $\Delta
_{S_{k}}(x,y)\geq k,$ which implies $S_{k}$ is a weak $k$-resolving set.
\end{proof}

To establish the weak $k$-metric dimension of all remaining trees, i.e. for
all trees $T\not =P_{n},S^{3}(l_{1},l_{2},l_{3}),$ we need the following
notation. Recall that given a root vertex $v\in V(T)$, the root degree of $v$
is denoted by $\ell(v),$ the lengths of threads hanging at $v$ are
$l_{1}(v)\leq\cdots\leq l_{\ell(v)}(v),$ and $P^{(j)}=vv_{1}^{(j)}\cdots
v_{l_{j}}^{(j)}$ is a thread hanging at $v$ for $j\in\lbrack\ell(v)].$
Finally, $R(T)$ denotes the set of all root vertices in a tree $T.$

For a root vertex $v$ of a tree $T$, let $x$ and $y$ be a pair of neighbors of
the root $v$ which belong to two different threads hanging at $v,$ say
$P^{(j)}$ and $P^{(k)}.$ Notice that $\Delta_{s}(x,y)=2$ if and only if
$s\in(V(P^{(j)})\cup V(P^{(k)}))\backslash\{v\},$ and $\Delta_{s}(x,y)=0$ for
any other $s\in V(T).$ Therefore, a necessary condition for $S_{k}\subseteq
V(G)$ to be a weak $k$-resolving set is that every pair of threads hanging at
a same root vertex $v$ contains at least $\left\lceil k/2\right\rceil $
vertices of $S_{k}$.

We denote by $S_{k}(v)$ the set of vertices from $S_{k}$ which belong to
threads hanging at $v.$ Since we want a smallest possible $S_{k}(v)$ such that
each pair of threads hanging at $v$ contains at least $\left\lceil
k/2\right\rceil $ vertices of $S_{k}(v)$, we will define $S_{k}(v)$ to contain
$\left\lceil k/4\right\rceil $ vertices of each thread. Notice that this is
only possible if a shortest thread hanging at $v$ contains at least
$\left\lceil k/4\right\rceil $ vertices, i.e. $\left\lceil k/4\right\rceil
\leq l_{1}(v),$ in which case we define
\[
S_{k}^{(1)}(v)=\{v_{i}^{(j)}:1\leq i\leq\left\lceil k/4\right\rceil ,1\leq
j\leq\ell(v)\}.
\]
Here, one has to be careful, since $S_{k}^{(1)}(v)$ contains $\left\lceil
k/4\right\rceil $ of each thread, this implies that every pair of threads
hanging at $v$ contains $2\left\lceil k/4\right\rceil $ vertices of
$S_{k}^{(1)}(v)$. It is easily verified that%
\[
2\left\lceil k/4\right\rceil =\left\{
\begin{array}
[c]{ll}%
\left\lceil k/2\right\rceil +1 & \text{if }k\equiv1,2\text{ }%
(\operatorname{mod}4),\\
\left\lceil k/2\right\rceil  & \text{if }k\equiv3,4\text{ }(\operatorname{mod}%
4).
\end{array}
\right.
\]
Hence, in the case of $k\equiv1,2$ $(\operatorname{mod}4)$, the set
$S_{k}^{(1)}(v)$ is not a smallest set with the desired property and we may
consider the set $S_{k}^{(1)}(v)\backslash\{v_{\left\lceil k/4\right\rceil
}^{(1)}\}$ instead.

If a shortest thread hanging at $v$ does not contain $\left\lceil
k/4\right\rceil $ vertices, i.e. $\left\lceil k/4\right\rceil >l_{1}(v),$ one
has to take more than $\left\lceil k/4\right\rceil $ vertices from longer
threads in order to obtain that every pair of threads hanging at $v$ contains
at least $\left\lceil k/2\right\rceil $ vertices of $S_{k}$. Namely, in this
case we define
\[
S_{k}^{(2)}(v)=\{v_{i}^{(1)}:i\in\lbrack l_{1}(v)]\}\cup\{v_{i}^{(j)}:1\leq
i\leq\left\lceil k/2\right\rceil -l_{1}(v),2\leq j\leq\ell(v)\}.
\]
It is obvious that in this case $P^{(1)}$ and $P^{(j)}$ contain precisely
$\left\lceil k/2\right\rceil $ vertices of $S_{k}^{(2)}(v),$ and $P^{(i)}$ and
$P^{(j)}$ for $i,j\geq2$ contain
\[
2(\left\lceil k/2\right\rceil -l_{1}(v))\geq2(\left\lceil k/2\right\rceil
-\left\lceil k/4\right\rceil +1)>\left\lceil k/2\right\rceil
\]
vertices of $S_{k}^{(2)}(v).$ We unite these definitions as follows
\[
S_{k}(v)=\left\{
\begin{array}
[c]{ll}%
S_{k}^{(1)}(v)\backslash\{v_{\left\lceil k/4\right\rceil }^{(1)}\} & \text{if
}\left\lceil k/4\right\rceil \leq l_{1}(v)\text{ and }k\equiv1,2\text{
}(\operatorname{mod}4),\\
S_{k}^{(1)}(v) & \text{if }\left\lceil k/4\right\rceil \leq l_{1}(v)\text{ and
}k\equiv3,4\text{ }(\operatorname{mod}4),\\
S_{k}^{(2)}(v) & \text{if }\left\lceil k/4\right\rceil >l_{1}(v).
\end{array}
\right.
\]
The set $S_{k}(v)$ is illustrated by Figure \ref{Fig_trees}.

In what follows we will establish that the union of $S_{k}(v)$ over all root
vertices $v$ is a smallest weak $k$-resolving set of a tree $T,$ but let us
first establish the cardinality of $S_{k}(v).$ For this purpose we define
\[
f_{k}(v)=\left\{
\begin{array}
[c]{ll}%
\left\lceil k/4\right\rceil \ell(v)-1 & \text{if }\left\lceil k/4\right\rceil
\leq l_{1}(v)\text{ and }k\equiv1,2\text{ }(\operatorname{mod}4),\\
\left\lceil k/4\right\rceil \ell(v) & \text{if }\left\lceil k/4\right\rceil
\leq l_{1}(v)\text{ and }k\equiv3,4\text{ }(\operatorname{mod}4),\\
l_{1}(v)+(\ell(v)-1)(\left\lceil k/2\right\rceil -l_{1}(v)) & \text{if
}\left\lceil k/4\right\rceil >l_{1}(v).
\end{array}
\right.
\]
and it is obvious that $\left\vert S_{k}(v)\right\vert =f_{k}(v).$ Finally,
for a tree $T\not =P_{n},S^{3}(l_{1},l_{2},l_{3})$ we define the set
\[
S_{k}(T)=%
{\textstyle\bigcup\nolimits_{v\in R(T)}}
S_{k}(v).
\]
In a situation where a tree $T$ is clear from the context, we will write only
$S_{k}.$ Since a pair of distinct root vertices $v,w\in R(T)$ satisfies
$S_{k}(v)\cap S_{k}(w)=\emptyset,$ we conclude that
\[
\left\vert S_{k}(T)\right\vert =%
{\textstyle\sum\nolimits_{v\in R(T)}}
f_{k}(v).
\]
Now, let us first formally establish that the set $S_{k}(T)$ has a desired
property that any two threads hanging at a same root vertex contain at least
$\left\lceil k/2\right\rceil $ vertices of $S_{k}(T).$

\begin{lemma}
\label{Obs_threads}Any two threads hanging at a same root vertex $v$ of a tree
$T\not =P_{n},S^{3}(l_{1},l_{2},l_{3})$ contain at least $\left\lceil
k/2\right\rceil $ vertices of the set $S_{k}(T).$ Moreover, the bound
$\left\lceil k/2\right\rceil $ is attained if at least one of the two threads
is a shortest thread $P^{(1)}$ hanging at $v.$
\end{lemma}

\begin{proof}
From the definition of $S_{k}(v)$ it follows that in the case of $\left\lceil
k/4\right\rceil \leq l_{1}(v)$ and $k\equiv1,2$ $(\operatorname{mod}4),$ every
pair of threads contains at least $2\left\lceil k/4\right\rceil -1=\left\lceil
k/2\right\rceil $ vertices of $S_{k}(v)$ and therefore of $S_{k}(T).$ The
bound is attained if one of the two threads is $P^{(1)},$ and we can choose
the notation of threads so that $P^{(1)}.$

In the case of $\left\lceil k/4\right\rceil \leq l_{1}(v)$ and $k\equiv3,4$
$(\operatorname{mod}4),$ every pair of threads contains precisely
$2\left\lceil k/4\right\rceil =\left\lceil k/2\right\rceil $ vertices of
$S_{k}(v)\subseteq S_{k}(T),$ so the bound is attained for every pair of
threads hanging at $v.$

Finally, in the case of $\left\lceil k/4\right\rceil >l_{1}(v),$ every pair of
threads which includes $P^{(1)}$ contains $l_{1}(v)+\left\lceil
k/2\right\rceil -l_{1}(v)=\left\lceil k/2\right\rceil $ vertices of
$S_{k}(v),$ and otherwise it contains $2(\left\lceil k/2\right\rceil
-l_{1}(v))$ vertices of $S_{k}(v)$ and we already established that
$2(\left\lceil k/2\right\rceil -l_{1}(v))>\left\lceil k/2\right\rceil $ in
this case.
\end{proof}

The above lemma implies the following result.

\begin{lemma}
\label{Cor_cardinality}If $T\not =P_{n},S^{3}(l_{1},l_{2},l_{3})$ is a tree on
$n$ vertices, then $\left\vert S_{k}(T)\right\vert \geq k.$
\end{lemma}

\begin{proof}
Notice that tree $T$ contains at least one root vertex, since $T\not =P_{n}.$
If $T$ contains precisely one root vertex $v$, then $T\not =S^{3}(l_{1}%
,l_{2},l_{3})$ implies $\ell(v)\geq4.$ Thus, $T$ contains at least two pairs
of threads hanging at the same root vertex, so according to Lemma
\ref{Obs_threads} we have $\left\vert S_{k}(T)\right\vert =\left\vert
S_{k}(v)\right\vert \geq k.$ If $T$ contains at least two root vertices, say
$v$ and $w,$ each of them by definition has at least two threads hanging at
it. Lemma \ref{Obs_threads} then implies $\left\vert S_{k}(T)\right\vert
\geq\left\vert S_{k}(v)\right\vert +\left\vert S_{k}(w)\right\vert \geq k.$
\end{proof}

So far we have defined the set $S_{k}(T)$ which certainly weakly $k$-resolves
each pair of vertices $x,y$ which are neighbors of a root $v$ of $T$ and
belong to two distinct threads hanging at $v.$ In the next theorem we will
establish that such a set also weakly $k$-resolves every pair of vertices in
$T,$ i.e. that it is a weak $k$-resolving set. Moreover, we will show that it
is a weak $k$-metric basis, obtaining thus the weak $k$-metric dimension of
$T.$

\begin{theorem}
If $T\not =P_{n},S^{3}(l_{1},l_{2},l_{3})$ is a tree on $n$ vertices, then
\[
\operatorname*{wdim}\nolimits_{k}(T)=%
{\textstyle\sum\nolimits_{v\in R(T)}}
f_{k}(v).
\]
Moreover, the set $S_{k}(T)$ is a weak $k$-metric basis of $T$ for $1\leq
k\leq\kappa(T).$
\end{theorem}

\begin{proof}
Since any root vertex $v$ of $T$ satisfies $\left\vert S_{k}(v)\right\vert
=f_{k}(v),$ and $S_{k}(T)$ is a union of pairwise vertex disjoint sets
$S_{k}(v)$ over all root vertices $v$ of $T,$ we have $\left\vert
S_{k}(T)\right\vert =%
{\textstyle\sum\nolimits_{v\in R(T)}}
f_{k}(v).$ Hence, it is sufficient to prove that $S_{k}(T)$ is a weak
$k$-metric basis of $T$ for $1\leq k\leq\kappa(T).$

Let us first establish that $S_{k}=S_{k}(T)$ is a weak $k$-resolving set of
$T.$ Let $x,y\in T$ be any pair of vertices at distance $d=d(x,y),$ and let
$P=u_{0}u_{1}\cdots u_{d}$ be the shortest path connecting $x$ and $y,$ where
$x=u_{0}$ and $y=u_{d}.$ If $T_{i}$ is a component of $T-E(P)$ which contains
$u_{i},$ Lemma \ref{Lemma_Riste} implies%
\[
\Delta_{S_{k}}(x,y)=\sum_{i=0}^{d}\left\vert d-2i\right\vert \left\vert
V(T_{i})\cap S_{k}\right\vert .
\]
If $d$ is odd, then $\left\vert d-2i\right\vert \geq1$ for any integer $i,$ so
Lemma \ref{Cor_cardinality} implies $\Delta_{S_{k}}(x,y)\geq\sum_{i=0}%
^{d}\left\vert V(T_{i})\cap S_{k}\right\vert =\left\vert S_{k}\right\vert \geq
k$. Hence, let us assume that $d$ is even. Notice that $\left\vert
d-2i\right\vert =0$ if and only if $i=d/2,$ and all other integers $i$ satisfy
$\left\vert d-2i\right\vert \geq2.$ Let $i^{\ast}\in\{0,\ldots,d\}$ be the
smallest integer such that $V(T_{i})$ contains a root vertex $v$ of $T.$ We
may assume that vertices $u_{i}$ of $P$ are denoted so that $i^{\ast}\leq
d/2.$

If $i^{\ast}<d/2,$ then the set $\bigcup\nolimits_{i=0}^{i^{\ast}}V(T_{i})$
contains all vertices of at least two threads hanging at $v.$ Hence, according
to Lemma \ref{Obs_threads} we have $\Delta_{S_{k}}(x,y)\geq2\sum
_{i=0}^{i^{\ast}}\left\vert V(T_{i})\cap S_{k}\right\vert \geq k.$ If
$i^{\ast}=d/2,$ then $u_{d/2}$ is a root vertex of $T$ and $x$ and $y$ belong
to two distinct threads hanging at $u_{d/2}.$ If we denote $I=\{0,\ldots
,d\}\backslash\{d/2\},$ the set $\bigcup\nolimits_{i\in I}V(T_{i})$ contains
all vertices of the two threads hanging at $u_{d/2}$ containing $x$ and $y,$
except the root $u_{d/2}$ itself. Lemma \ref{Obs_threads} again implies
$\Delta_{S_{k}}(x,y)\geq2\sum_{i\in I}\left\vert V(T_{i})\cap S_{k}\right\vert
\geq k.$

Hence, we have proven that $S_{k}$ is a weak $k$-resolving set of $T.$ It
remains to prove that $S_{k}$ is a smallest such set. Suppose to the contrary
that there exists a weak $k$-resolving set $S$ of $T$ such that $\left\vert
S\right\vert <\left\vert S_{k}\right\vert .$ Since $S_{k}$ contains only
vertices which belong to threads hanging at root vertices, the assumption
$\left\vert S\right\vert <\left\vert S_{k}\right\vert $ implies that there
exists a root vertex $v$ in $T$ and a pair of threads $P_{1}$ and $P_{2}$
hanging at $v$ such that $V(P_{1})\cup V(P_{2})$ contains less vertices of $S$
than of $S_{k}$ and $P_{1}$ is a shortest thread hanging at $v.$ Lemma
\ref{Obs_threads} implies
\[
\left\vert (V(P_{1})\cup V(P_{2}))\cap S\right\vert \leq\left\vert
(V(P_{1})\cup V(P_{2}))\cap S_{k}\right\vert -1=\left\lceil k/2\right\rceil
-1<k/2.
\]
Let $x$ and $y$ be vertices of $P_{1}$ and $P_{2}$, respectively, adjacent to
$v.$ Notice that $\Delta_{s}(x,y)=2$ for any $s\in S$ contained in $P_{1}$ or
$P_{2},$ otherwise $\Delta_{s}(x,y)=0.$ Hence, $\Delta_{S}(x,y)<2k/2=k,$ which
contradicts $S$ being a weak $k$-resolving set.
\end{proof}

\subsection{Grid graphs}

It remains to consider grid graphs, for which we also established weak
dimensionality in Section \ref{sec:kappa}. Recall that $\kappa(P_{q}\Box
P_{r})=2q+2r-4$ according to Theorem \ref{Prop_grid_kapa}. The weak $k$-metric
dimension of any grid graph $P_{q}\Box P_{r}$ is given by the following theorem.

\begin{figure}[ptbh]
\begin{center}
\includegraphics[scale=0.6]{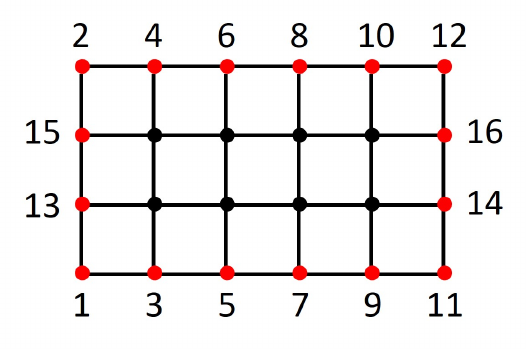}
\end{center}
\caption{The figure shows the grid graph $P_{6}\Box P_{4}$ and the values of
the $f(v)$ for each vertex $v$ from the border $R$ of the grid. Hence, in this
grid we have for example $S_{3}=S_{4}=\{v\in R:f(v)\leq4\}$ and $S_{5}%
=S_{6}=\{v\in R:f(v)\leq6\}.$}%
\label{Fig_grid_f}%
\end{figure}

\begin{theorem}
If $q,r\geq2$ are integers, then
\[
\operatorname*{wdim}\nolimits_{k}(P_{q}\Box P_{r})=\left\{
\begin{array}
[c]{ll}%
k & \mbox{if $k$ is even},\\
k+1 & \mbox{if $k$ is odd},
\end{array}
\right.
\]
for every $1\leq k\leq\kappa(P_{q}\Box P_{r})=2q+2r-4.$
\end{theorem}

\begin{proof}
Notice that $2\left\lceil k/2\right\rceil =k$ for even $k,$ and $2\left\lceil
k/2\right\rceil =k+1$ for odd $k.$ Hence, it is sufficient to find a weak
$k$-metric basis $S$ of $P_{q}\Box P_{r}$ with $\left\vert S\right\vert
=2\left\lceil k/2\right\rceil .$ In order to find such a set $S,$ we first
need several notions.

Let $R$ be the set of border vertices of the grid $G,$ i.e. let
\begin{align*}
R_{1} &  =\{(u_{i},v_{1}),(u_{i},v_{r}):i\in\lbrack q]\},\\
R_{2} &  =\{(u_{1},v_{j}),(u_{q},v_{j}):j\in\lbrack r-1]\backslash\{1\}\},
\end{align*}
and $R=R_{1}\cup R_{2}.$ Notice that $\left\vert R\right\vert =2q+2r-4.$ We
define a function $f:R\rightarrow\lbrack2q+2r-4]$ by requiring that
$f(x)<f(y)$ for every $x\in R_{1}$ and $y\in R_{2},$ and on $R_{1}$ by
\[
f(u_{i},v_{1})<f(u_{i},v_{r})<f(u_{i+1},v_{1})<f(u_{i+1},v_{r})
\]
for every $i\in\lbrack q-1],$ and also on $R_{2}$ by
\[
f(u_{1},v_{j})<f(u_{q},v_{j})<f(u_{1},v_{j+1})<f(u_{q},v_{j+1})
\]
for every $j\in\{2,\ldots,r-2\}.$ The function $f$ is illustrated by Figure
\ref{Fig_grid_f}. Notice that $f$ is a uniquely defined bijection. For every
$k\leq2q+2r-4,$ we define a set $S_{k}=\{x\in R:f(x)\leq2\left\lceil
k/2\right\rceil \}.$ Since $f$ is bijection, we have $\left\vert
S_{k}\right\vert =2\left\lceil k/2\right\rceil .$ We have to show that $S_{k}$
is a weak $k$-metric basis of $P_{q}\Box P_{r}.$

Let us first prove that $S_{k}$ is a weak $k$-resolving set of $P_{q}\Box
P_{r}.$ For any pair of vertices $x,y\in V(G)$ such that $d(x,y)$ is odd,
Lemma \ref{Lemma_riste2} implies $\Delta_{S_{k}}(x,y)\geq\left\vert
S_{k}\right\vert \geq k.$ So, let us assume that $d(x,y)$ is even. Let
$x=(u_{i_{1}},v_{j_{1}})$ and $y=(u_{i_{2}},v_{j_{2}}),$ and we denote $\Delta
i=i_{2}-i_{1}$ and $\Delta j=j_{2}-j_{1}$ where we may assume $0\leq\Delta
j\leq\Delta i.$ Similarly as in the proof of Theorem \ref{Prop_grid_kapa}, we
define the set $S_{0}(x,y)$ to consist of all vertices $s$ of the grid such
that $\Delta_{s}(x,y)=0$ (see Figure \ref{Fig_S0}).

Let us establish the property of $S_{k}$ that at least half of its vertices
are not contained in $S_{0}(x,y).$ Assume first that $\Delta j=0.$ Recall from
the proof of Theorem \ref{Prop_grid_kapa} that in this case the set
$S_{0}(x,y)$ contains only two vertices from the border $R,$ namely
$(u_{i},v_{1})$ and $(u_{i},v_{r})$ for $i=i_{1}+\Delta i/2,$ see Figure
\ref{Fig_S0}.a). Notice that $\Delta j=0$ together with $d(x,y)$ being even
imply $\Delta i\geq2,$ so $i=i_{1}+\Delta i/2\geq2.$ Since $S_{k}$ contains
vertices $(u_{1},v_{1})$ and $(u_{1},v_{r}),$ we conclude that the property
holds. Assume next that $\Delta j>0.$ Recall again from the proof of Theorem
\ref{Prop_grid_kapa} that (in this case) the set $S_{0}(x,y)$ contains at most
one vertex of the pair $(u_{i},v_{1})$ and $(u_{i},v_{r}),$ and the same holds
for the pair $(u_{1},v_{j})$ and $(u_{q},v_{j}),$ which is also evident from
Figures \ref{Fig_S0}.b) and \ref{Fig_S0}.c). Since $S_{k}$ contains either
none or both vertices of these pairs, we again conclude that at least half of
the vertices of $S_{k}$ are not contained in $S_{0}(x,y),$ and the property is established.

Further, $d(x,y)$ being even implies $\Delta_{s}(x,y)\geq2$ for every $s\in
V(G)\backslash S_{0}(x,y).$ Combining this with the fact that at least half of
the vertices of $S_{k}$ are not contained in $S_{0}(x,y),$ we obtain
$\Delta_{S_{k}}(x,y)\geq2\left\vert S_{k}\right\vert /2\geq k.$ This
establishes that $S_{k}$ is a weak $k$-resolving set.

It remains to prove that $S_{k}$ is a smallest such set. In the case of even
$k,$ this is a consequence of Proposition \ref{Lemma_auxiliary}. In the case
of odd $k,$ we have to show that a set $S\subseteq V(G)$ such that $\left\vert
S\right\vert =k$ cannot be a weak $k$-resolving set. So, let $S$ be such a
set, and let us define four subsets of the border $R$ of a grid $G$ as follows%
\begin{align*}
R_{1,1} &  =\{(u_{i},v_{1}),(u_{1},v_{j}):i>1\text{ and }j>1\},\\
R_{q,1} &  =\{(u_{i},v_{1}),(u_{q},v_{j}):i<q\text{ and }j>1\},\\
R_{1,r} &  =\{(u_{i},v_{r}),(u_{1},v_{j}):i>1\text{ and }j<r\},\\
R_{q,r} &  =\{(u_{i},v_{r}),(u_{q},v_{j}):i<q\text{ and }j<r\}.
\end{align*}
Notice that for $(i,j)\in\{(1,1),(q,1),(1,r),(q,r)\},$ the pair $x,y$ of
neighbors of $(u_{i},v_{j})$ is distinguished only by vertices of $R_{i,j},$
moreover $\Delta_{s}(x,y)=2$ for every $s\in R_{i,j}.$ Hence, if any of the
sets $R_{i,j}$ contains less than $k/2$ vertices of $S,$ then $\Delta
_{S}(x,y)<2k/2=k,$ which implies $S$ is not a weak $k$-resolving set. So, let
us assume that sets $R_{i,j}$ contain at least $k/2$ vertices of $S$. Since
$k$ is odd and the number of vertices integer, this further implies that
$\left\vert R_{i,j}\cap S\right\vert \geq(k+1)/2.$ Let us define
$\lambda_{i,j}=1$ if $(u_{i},v_{j})\in S,$ and $\lambda_{i,j}=0$ otherwise. We
have%
\begin{align*}
\left\vert S\right\vert  &  =\left\vert R_{1,1}\cap S\right\vert +\left\vert
R_{q,r}\cap S\right\vert +\lambda_{1,1}+\lambda_{q,r}-\lambda_{1,r}%
-\lambda_{q,1}\\
&  \geq2\frac{k+1}{2}+\lambda_{1,1}+\lambda_{q,r}-\lambda_{1,r}-\lambda_{q,1}.
\end{align*}
Similarly, we have%
\begin{align*}
\left\vert S\right\vert  &  =\left\vert R_{1,r}\cap S\right\vert +\left\vert
R_{q,1}\cap S\right\vert +\lambda_{1,r}+\lambda_{q,1}-\lambda_{1,1}%
-\lambda_{q,r}\\
&  \geq2\frac{k+1}{2}+\lambda_{1,r}+\lambda_{q,1}-\lambda_{1,1}-\lambda_{q,r}.
\end{align*}
We conclude that $2\left\vert S\right\vert \geq2(k+1),$ i.e. $\left\vert
S\right\vert \geq k+1,$ a contradiction.
\end{proof}

\section{Concluding remarks and further work}

In this paper we generalize the classical notion of metric dimension by
introducing the weak $k$-metric dimension. In the classical metric dimension,
a set $S$ is a resolving set if every pair of vertices $x,y\in V(G)$ is
distinguished by at least one $s\in S.$ This can be written as $\sum_{s\in
S}\left\vert d(x,s)-d(y,s)\right\vert \geq1.$ Sometimes, this distinguishing
may not be strong enough, so the condition can be reformulated as $\sum_{s\in
S}\left\vert d(x,s)-d(y,s)\right\vert \geq k,$ which means that every pair of
vertices is distinguished either by a large number of sensors $s\in S$ with a
small distance difference or by a small number of sensors $s\in S$ with larger
distance difference. In this paper, we provided the integer linear programming
model for this variation of metric dimension. Such model can be used to make
implementations for computing or approximating the value of the weak
$k$-metric dimension of other more complicated families of graphs, in which
exact values are not yet known. The influence of the existence of true or
false twins in a graph to its weak $k$-metric dimension is also considered.

Next, we compared the weak $k$-metric dimension and the $k$-metric dimension,
and established that $\dim(G)\leq\operatorname*{wdim}_{k}(G)\leq\dim_{k}(G).$
Here, we find the following problem interesting.

\begin{problem}
Determine the graphs $G$ with $\operatorname*{wdim}_{k}(G)=\dim_{k}(G)$ for a
given $k$.
\end{problem}

Another interesting comparison is the one concerning the weak $k$-metric
dimension and the local $k$-metric dimension, where we established
$\operatorname{ldim}_{k}(G)\leq\operatorname*{wdim}_{k}(G)$. Here, the
following natural problem arises.

\begin{problem}
Determine the graphs $G$ which satisfy $\operatorname{ldim}_{k}%
(G)=\operatorname*{wdim}_{k}(G)$ for a given $k$.
\end{problem}

We also considered the question of the maximum value of $k$ for which the weak
$k$-metric dimension is defined for some simple graphs such as path, cycle,
star, complete graph and complete bipartite graph, and also grid graphs and
trees. For the same families of graphs, we established the exact value of the
weak $k$-metric dimension for every feasible $k.$

For the classical metric dimension the variants such as edge metric dimension
\cite{TratnikEdge} and mixed metric dimension \cite{Kelm} were introduced,
defined as the cardinality of a smallest set $S\subseteq V(G)$ which
distinguishes all pairs of edges and all pairs of vertices and edges,
respectively. A weak version for both variants can be introduced in a similar
manner. First, we define the distance between a vertex $u$ and an edge $e=vw$
of a graph $G$ by $d(u,e)=\min\{d(u,v),d(u,w)\}.$

\begin{definition}
For a graph $G,$ a \emph{weak edge }$k$-\emph{metric dimension}
$\operatorname{wedim}_{k}(G)$ (resp. \emph{weak mixed }$k$-\emph{metric
dimension} $\operatorname{wmdim}_{k}(G)$) is the cardinality of a smallest set
$S\subseteq V(G)$ such that $\sum_{s\in S}\left\vert d(x,s)-d(y,s)\right\vert
\geq k$ for every $x,y\in E(G)$ (resp. for every $x,y\in V(G)\cup E(G)$).
\end{definition}

For both weak edge and weak mixed $k$-metric dimension, an integer linear
programming model can be easily formulated by slightly modifying the model
(\ref{For_ILP}). Namely, for a set $S\subseteq V(G)$ we define an integer
variable $x_{i}=1$ if $v_{i}\in S$, and $x_{i}=0$ otherwise. The problem of
finding the weak edge $k$-metric dimension graph $G$ is formulated as
\[%
\begin{tabular}
[c]{ll}
& $\min\sum_{i=1}^{n}x_{i}\medskip$\\
s.t. & $\sum\limits_{i=1}^{n}\left\vert d(e_{a},v_{i})-d(e_{b},v_{i}%
)\right\vert \cdot x_{i}\geq k$ for any $e_{a},e_{b}\in E(G),$\\
& $x_{i}\in\{0,1\}$ for $1\leq i\leq n.$%
\end{tabular}
\ \ \ \ \ \ \
\]
To formulate the ILP for the weak mixed $k$-metric dimension, one just needs
to expand the first constraint to all pairs from $V(G)\cup E(G),$ namely%
\[%
\begin{tabular}
[c]{ll}
& $\min\sum_{i=1}^{n}x_{i}\medskip$\\
s.t. & $\sum\limits_{i=1}^{n}\left\vert d(m_{a},v_{i})-d(m_{b},v_{i}%
)\right\vert \cdot x_{i}\geq k$ for any $m_{a},m_{b}\in V(G)\cup E(G),$\\
& $x_{i}\in\{0,1\}$ for $1\leq i\leq n.$%
\end{tabular}
\
\]
Regarding these dimensions, the following problems seem interesting.

\begin{problem}
Characterize families of graphs for which weak vertex and weak edge $k$-metric
dimensions differ.
\end{problem}

\noindent Similar characterizations have been done for the clasical vertex
metric dimension and its edge variant, so it would be interesting to compare
the results.

Finally, since usually the problems concerning computing any metric dimension
related parameter are NP-hard, it is probably not surprising that finding the
weak $k$-metric dimension is such too. Nevermind, the following problem is
worthy of considering.

\begin{problem}
Which is the complexity of computing the weak $k$-metric dimension of graphs?
Can this problem be polynomially approximated?
\end{problem}

\bigskip

\bigskip\noindent\textbf{Acknowledgments.}~~I.P. was partially supported by
the Slovenian Research Agency program No. P1-0297. J.S. and R.\v{S}. have been
partially supported by the Slovenian Research Agency ARRS program\ P1-0383 and
ARRS project J1-3002, J.S. also acknowledges the support of Project
KK.01.1.1.02.0027, a project co-financed by the Croatian Government and the
European Union through the European Regional Development Fund - the
Competitiveness and Cohesion Operational Programme. I.G.Y. has been partially
supported by the Spanish Ministry of Science and Innovation through the grant PID2019-105824GB-I00.


\begin{thebibliography}{99}                                                                                               %


\bibitem {Blumenthal}L. M. Blumenthal, Theory and applications of distance
geometry, Oxford University Press (1953).

\bibitem {Claverol}M.~Claverol, A.~Garc\'{\i}a, G.~Hernandez, C.~Hernando,
M.~Maureso, M.~Mora, J.~Tejel, Metric dimension of maximal outerplanar graphs,
\emph{Bull.\ Malays.\ Math.\ Sci.\ Soc.}\ \textbf{44} (2021) 2603--2630.

\bibitem {Corregidor}S.~G.~Corregidor, A.~Mart\'{\i}nez-P\'{e}rez, A note on
$k$-metric dimensional graphs, \emph{Discrete Appl.\ Math.}\ \textbf{289}
(2021) 523--533.

\bibitem {Estrada-Moreno2013}A.~Estrada-Moreno, J.~A.
Rodr\'{\i}guez-Vel\'{a}zquez, I.~G. Yero, The $k$-metric dimension of a graph,
\emph{Appl. Math. Inf. Sci.} \textbf{9~(6)} (2015) 2829--2840.

\bibitem {Geneson}J.~Geneson, S.~Kaustav, A.~Labelle, Extremal results for
graphs of bounded metric dimension, \emph{Discrete Appl.\ Math.}\ \textbf{309}
(2022) 123--129.

\bibitem {Hakanen}A.~Hakanen, V.~Junnila, T.~Laihonen, I.~G.~Yero, On vertices
contained in all or in no metric basis, \emph{Discrete Appl.\ Math.}%
\ \textbf{319} (2022) 407--423.

\bibitem {Harary1976}F.~Harary, R.~A. Melter, On the metric dimension of a
graph, \emph{Ars Comb.} \textbf{2} (1976) 191--195.

\bibitem {Kelm}A. Kelenc, D. Kuziak, A. Taranenko, I. G. Yero, Mixed metric
dimension of graphs, \emph{Appl. Math. Comput.} \textbf{314} (1) (2017) 429--438.

\bibitem {TratnikEdge}A. Kelenc, N. Tratnik, I. G. Yero, Uniquely identifying
the edges of a graph: the edge metric dimension, \emph{Discrete Appl. Math.}
\textbf{251} (2018) 204--220.

\bibitem {Klavzar}S.~Klav\v{z}ar, F.~Rahbarnia, M.~Tavakoli, Some binary
products and integer linear programming for k-metric dimension of graphs,
\emph{Appl.\ Math.\ Comput.}\ \textbf{409} (2021) 126420.

\bibitem {Kuziak}D. Kuziak, I. G. Yero, Metric dimension related parameters in
graphs: A survey on combinatorial, computational and applied results. (2021)
arXiv preprint arXiv:2107.04877 [math.CO].

\bibitem {Mashkaria}S.~Mashkaria, G.~\'{O}dor, P.~Thiran, On the robustness of
the metric dimension of grid graphs to adding a single edge, \emph{Discrete
Appl.\ Math.}\ \textbf{316} (2022) 1--27.

\bibitem {Okamoto}F. Okamoto, B. Phinezy, P. Zhang, The local metric dimension
of a graph, \emph{Math. Bohem.} \textbf{135(3)} (2010) 239--255.

\bibitem {Sedlar}J.~Sedlar, R.~\v{S}krekovski, Metric dimensions vs.
cyclomatic number of graphs with minimum degree at least two,
\emph{Appl.\ Math.\ Comput.}\ \textbf{427} (2022) 127147.

\bibitem {Slater1975}P.~J. Slater, Leaves of trees, \emph{Congr. Numer.}
\textbf{14} (1975) 549--559.

\bibitem {Tillquist}R. C. Tillquist, R. M. Frongillo, M. E. Lladser, Getting
the lay of the land in discrete space: A survey of metric dimension and its
applications. (2021) arXiv:2104.07201 [math.CO].

\bibitem {Wu}J.~Wu, Jian, L.~Wang, W.~Yang, Learning to compute the metric
dimension of graphs, \emph{Appl.\ Math.\ Comput.}\ \textbf{432} (2022) 127350.
\end{thebibliography}
\end{document}